\documentclass{amsart}

\usepackage{enumerate}
\usepackage[utf8]{inputenc}
\usepackage{amsmath}
\usepackage{amsthm}
\usepackage{amsfonts}
\usepackage{amssymb}
\usepackage{color}

\theoremstyle{definition}
\newtheorem{definition}{Definition}[section]

\theoremstyle{definition}
\newtheorem{theorem}{Theorem}[section]
\newtheorem{corollary}{Corollary}[section]
\newtheorem{proposition}{Proposition}[section]
\newtheorem{lema}{Lemma}[section]

\theoremstyle{remark}
\newtheorem{remark}{Remark}[section]

\begin{document}	
\title[Radial Foliations]{Radial Foliations in dimension Three}
\author{Felipe Cano}
\author{Beatriz Molina-Samper}
\address{Dpto. \'Algebra, An\'alisis Matem\'atico, Geometr\'ia y Topolog\'ia and Instituto de Matem\'aticas de la Universidad de Valladolid, Facultad de Ciencias, Universidad de Valladolid, Paseo de Bel\'en, 7, 47011 Valladolid, Spain}
\email{fcano@uva.es, beatriz.molina@uva.es}
\thanks{Both authors are supported by Spanish  Ministerio de Ciencia, under the Project PID2022-139631NB-100.}			

\date{\today}

\begin{abstract}
	Radial germs of holomorphic foliations in dimension two have a characteristic property:  they are the only singular foliations whose reduction of singularities has no singular points. We also know that they are desingularized by a single dicritical blowing-up.
Let us say that a foliated space $((\mathbb C^3,\mathbf 0),E,\mathcal F)$ is {\em almost radial} when it has a reduction of singularities without singular points; it will be ``radial'' under a certain additional condition on the morphism of reduction of singularities. We show that the radial condition corresponds to the ``open book'' situation. We end the paper with a discussion on the general almost radial case.
\end{abstract}
\maketitle
\tableofcontents
\section{Introduction}
The reduction of singularities of a function, in the sense that we want to obtain the normal crossings property for the total transform, produces systematically singular points in the corners of the total transform. This is so unless we perform no blowing-ups because the function is already a local coordinate.

In the case of foliations over an ambient space of dimension two the situation is not exactly the same one.  We can look at the foliation on $({\mathbb C}^2,\mathbf 0)$ given by the $1$-form
$$
xdy-ydx
$$
whose invariant curves are the lines through the origin. The origin is of course a singular point, but if we perform the corresponding quadratic blowing-up, we obtain a foliation without singularities, that is transverse to the exceptional divisor.

In fact, among the foliations over $(\mathbb C^2,\mathbf 0)$, the above (radial) one is the only singular foliation that can be totally desingularized with finite sequences of blowing-ups.

This work is devoted to analyse the natural similar question in ambient dimension three. Namely, what kind of codimension one foliations in an ambient dimension three may be totally desingularized by means of a sequence of blowing-ups?

The first example in dimension three is the ``open book foliation'' given in coordinates $x,y,z$ by the one form
$$
ydz-zdy.
$$
This foliation is a cylinder over the radial foliation in dimension two.

The main result in this paper is that all radial foliated spaces, i.e. that are totally desingularized by a ``controlled'' sequence of blowing-ups, are exactly ``open book foliated spaces''. We give some examples of unexpected desingularizable foliated spaces, where the ``controlled condition'' fails.

A word about the definition of foliated space. The ambient spaces are not just regular varieties $M$, but  a pair $(M,E)$ where $E$ is a given normal crossings divisor, that is naturally transformed by the centers of blowing-up. Hence, a foliated space is a triple
$$
(M,E,\mathcal F),
$$
where $\mathcal F$ is a codimension one singular foliation on $(M,E)$. The desired regular simple points are non-singular for $\mathcal F$, but we also ask $\mathcal F$ to have normal crossings with $E$ at them.

A systematic description of radial foliated spaces foliations  in dimension two is necessary. In the germified case, in the case of the projective plane and in the case of Hirzebruch surfaces. We also present the notion of ``Hirzebruch tube'' in ambient dimension three, that is a key tool for our arguments.

We call ``almost radial foliated spaces'' those that can be totally desingularized, without asking the control condition on the resolution sequence. This paper ends by showing that almost radial spaces have (infinitely many) invariant surfaces. The statement is related with Local Brunella Alternative, see \cite{Can-R-S}, \cite{Can-R} and \cite{Cer}. In the cases when the Alternative is known to be true, the so-called ``partial separatrices'' are useful in the proofs. A reduction of singularities without partial separatrices should have only simple singularities of the type ``corner'' (see \cite{Can},\cite{Can-C},\cite{Can-MS})). We think that this kind of reduction of singularities corresponds in fact to the almost radial case.

\section{Blowing-ups of Codimension One Foliated Spaces}
Let $M$ be a non-singular complex analytic space. An {\em ambient space over $M$} is a pair $(M,E)$, where $E$ is a normal crossings divisor of $M$.
A {\em foliated space $\mathcal M$ over $(M,E)$} is a triple
$$
\mathcal M=(M,E,\mathcal F),
$$
where $\mathcal F$ is a codimension one singular holomorphic foliation over $M$.

Consider a codimension one singular holomorphic foliation $\mathcal F$ over $M$ and a point $P\in M$. We recall that $\mathcal F$ is locally generated at $P$ by a germ of holomorphic $1$-form
$$
\omega=\sum_{i=1}^na_idx_i,\quad a_i\in \mathcal O_{M,P},\, i=1,2,\ldots,n,
$$
satisfying the Frobenius integrability condition $\omega\wedge d\omega=0$ and such that the coefficients $a_i$ are without common factor. We say that $P$ is a {\em singular point of $\mathcal F$} when $a_i(P)=0$,
for any $i=1,2,\ldots,n$. Let us denote by $\operatorname{Sing}(\mathcal F)$ the set of singular points of $\mathcal F$, endowed with its structure of reduced closed analytic subspace of $M$. Note that $\operatorname{Sing}(\mathcal F)$ has codimension at least two in $M$.

A reduced and irreducible closed analytic subspace $Y$ of $M$ is said to be {\em invariant for $\mathcal F$} when $i^*\omega=0$, denoting by $i:Y_{\operatorname{reg}}\rightarrow M$ the inclusion of the regular part $Y_{\operatorname{reg}}$ of $Y$ in $M$. This property can be read locally at a single regular point of $Y$.

Consider a foliated space $\mathcal M=(M,E,\mathcal F)$ and an irreducible component $D$ of $E$. We use the terminology {\em dicritical component} for saying that $D$ is not invariant and {\em non-dicritical component} for saying that $D$ is invariant. Thus, we can decompose the divisor $E$ into two normal crossings divisors
$$
E=E_{\operatorname{dic}}\cup E_{\text{inv}},
$$
where $E_{\operatorname{dic}}$ is the union of the dicritical components of $\mathcal M$ and $E_{\text{inv}}$ is the union of the non-dicritical components of $\mathcal M$.

A blowing-up of ambient spaces $\pi: (M',E')\rightarrow (M,E)$ is given by a blowing-up $\pi:M'\rightarrow M$ with center a non-singular closed irreducible analytic subspace $Y\subset M$ having normal crossings with the divisor $E$. The divisor $E'$ is defined to be $E'=\pi^{-1} (E\cup Y)$. We hope that there is no confusion with the double use of the notation $\pi$.
\begin{remark}
In the case when $Y$ has codimension exactly one, the blowing-up $\pi$ has the following sense: the morphism $\pi: M' \to M$ is the identity morphism and the new divisor $E'$ is given by $E'=E\cup Y$.
\end{remark}
Given a blowing-up $\pi: (M',E')\rightarrow (M,E)$ of ambient spaces centered at $Y\subset M$ and a foliated space $\mathcal M=(M,E,\mathcal F)$, we have a transformation of foliated spaces
$$
\pi: \mathcal M'=(M',E',\mathcal F')\rightarrow (M,E,\mathcal F)=\mathcal M
$$
where $\mathcal F'$ is locally generated by $\pi^*\omega$ divided by an appropriate power of a local equation of the exceptional divisor $\pi^{-1}(Y)$. The transformation $\pi:\mathcal M'\rightarrow \mathcal M$ will be called an {\em admissible blowing-up} when the center $Y$ is invariant for $\mathcal F$.

We say that the blowing-up $\pi$ is {\em dicritical} when the exceptional divisor $\pi^{-1}(Y)$ is dicritical for $\mathcal M'$; otherwise, we say that $\pi$ is {\em non-dicritical}.

\begin{remark}
In general, we ask the centers to be invariant since otherwise the blowing-up breaks the dynamics. For instance, the blowing-up of the $x$-axis in $(\mathbb C^3,0)$ with respect to the regular foliation given by $dx=0$.
\end{remark}

\subsection{Simple Regular Points}
The reduction of singularities of a foliated space consists in finding a finite sequence of admissible blowing-ups in such a way that the last transformed foliated space would have only {\em simple points}. For details, the reader can look at \cite{Can,Can-C,Can-MS}.

Simple points may be singular for $\mathcal F$ or not. In this work we are not interested in giving a complete description of simple singular points, we just precise the definition for non-singular simple points.
\begin{definition}
Consider a foliated space $\mathcal M=(M,E,\mathcal F)$ and a point $P\in M$ such that $P\notin \operatorname{Sing}(\mathcal F)$. We say that $P$ is a {\em simple regular point for $\mathcal M$} when $\mathcal F$ and $E$ have normal crossings at $P$ in the following sense: there is a local coordinate system $(x_1,x_2,\ldots,x_n)$ such that $\mathcal F$ is locally given at $P$ by $dx_{1}=0$ and $E\subset (x_1x_2\cdots x_n=0)$.
\end{definition}
\begin{remark}\label{rk:nodicritical corners}
If $P\in M$ is a simple regular point for $\mathcal M$, then we have two possibilities:
\begin{enumerate}
\item The divisor $E_{\operatorname{inv}}$ has exactly one component through $P$.
\item All the irreducible components of  $E$ through $P$ are dicritical ones. In this case, the number of  irreducible components of  $E$ through $P$ is strictly smaller than the dimension $n=\dim M$. In other words, the point $P$ cannot be a ``corner'' of dicritical components.
\end{enumerate}
\end{remark}

\subsection{Indestructible Singularities}
We are interested in points where the foliation has a holomorphic first integral.

We say that a codimension one holomorphic singular  foliation $\mathcal F$ over $M$ has a {\em local holomorphic first integral at a point $P$} if $\mathcal F$ is generated locally at $P$ by a $1$-form of the type
$$
\omega=df/h,
$$
where $f$ and $h$ are germs of holomorphic functions at $P$.

Note that $\mathcal F$ has always holomorphic first integral at non-singular points.

\begin{proposition}
\label{prop:indestructible}
Let $\pi:\mathcal M'\rightarrow\mathcal M$ be an admissible blowing-up with center $Y\subset M$ and consider a point $P\in Y$ where $\mathcal F$ has a local first integral at $P$. Then, the transform $\mathcal F'$ has local first integral at all the points $P'\in \pi^{-1}(P)$. Moreover, if the codimension of $Y$ is bigger than or equal to two, there is at least a singular point in $\pi^{-1}(P)$.
\end{proposition}
\begin{proof}
The first part comes essentially from the commutativity of the pull-back of differential forms given by
$$
\pi^* df= d(f \circ \pi).
$$
For the second part, up to adding a constant to the first integral $f$, we may assume that $f(P)=0$. Moreover, the fact $Y$ is invariant implies that $Y$ is contained in the hypersurface $H=(f=0)$. In this situation, the exceptional divisor $\pi^{-1}(Y)$ is non-dicritical for $\mathcal M'$ and all the points of $\pi^{-1}(Y)$ belonging to the strict transform $H'$ of $H$ are necessarily singular points for $\mathcal F'$. When $Y$ has codimension greater than or equal to two, we see that $\pi^{-1}(P)\cap H' \ne \emptyset$.
\end{proof}

As a consequence, the singular points where $\mathcal F$ has local first integral are ``indestructible'' under admissible blowing-ups. Moreover, if we start with a non-singular point and we perform and admissible blowing-up with center of codimension at least two, then we generate at least one of that ``indestructible'' singularities.

\subsection{Two-Dimensional Index-Persistent Points}
\label{Two-Dimensional Index-Persistent Points}
In this subsection we present another type of ``indestructible'' singularities in ambient dimension two.

\begin{definition}
Let $\mathcal M=(M,E,\mathcal F)$ be a foliated space of dimension two and a point $P\in M$. We say that $P$ is an {\em index-persistent point} for $\mathcal M$ if there is a non-singular curve $\Gamma \subset M$ invariant for $\mathcal F$ with $P\in \Gamma$ and such that the Camacho-Sad index at $P$ of $\mathcal F$ with respect to $\Gamma$ has strictly negative real part.
\end{definition}
An index-persistent point is necessarily a singular point for $\mathcal F$.
\begin{remark}
	\label{remark:index-persistent}
Take a point $P\in M$ and a curve $\Gamma$ invariant for $\mathcal F$ such that $P\in \Gamma$. Let $P'$ be the point of $\pi^{-1}(P)$ defined by the strict transform $\Gamma'$ of $\Gamma$ by $\pi$. We know (see \cite{Cam-S,Can-C-D}) that the Camacho-Sad index of $\mathcal F'$ at $P'$ with respect to $\Gamma'$ is exactly of one unit less that the Camacho-Sad index of $\mathcal F$ at $P$ with respect to $\Gamma$.
\end{remark}
As a consequence of the previous remark, we have the following result:
\begin{proposition} \label{prop:camachosad}
Let $\pi:\mathcal M'\to\mathcal M$ be an admissible blowing-up of two-dimensional foliated spaces. If there is an index-persistent point for $\mathcal M$, then there is also an index-persistent point for $\mathcal M'$. Hence $\mathcal M'$ has at least one singular point.
\end{proposition}

\subsection{Radial Foliated Spaces}
\label{Radial Foliated Spaces}
We end this section by defining the notion of {\em radial foliated space}.
\begin{definition}
\label{def:almostradialfoliatedspace}
Consider a foliated space $\mathcal M=(M,E,\mathcal F)$. We say that $\mathcal M$ is an {\em almost radial foliated space} if and only if there is a finite sequence of admissible blowing-ups
\begin{equation}
	\label{eq:desingularizationsequencegeneral}
	\mathcal S:\quad\quad
	\mathcal M=\mathcal M_0\stackrel{\pi_1}{\leftarrow}
	\mathcal M_1
	\stackrel{\pi_2}{\leftarrow}
	\cdots
	\stackrel{\pi_N}{\leftarrow}
	\mathcal M_N=(M_N,E^N,\mathcal F_N)
\end{equation}
such that all the points in $M_N$ are simple regular points for $\mathcal M_N$. We say that the sequence $\mathcal S$ in Equation \eqref{eq:desingularizationsequencegeneral}  is a {\em resolution sequence for $\mathcal M$}.
\end{definition}
The case $N=0$ corresponds to an empty resolution sequence and in this case we have that $\mathcal M$ has no singular points and all the points are regular simple points.
The following features are direct from the definition:
\begin{itemize}
\item The intermediate steps $\mathcal M_j$ are also almost radial foliated spaces, for $j=0,1,2,\ldots,N$. Indeed, a resolution sequence for $\mathcal M_j$ is given by the blowing-ups $\pi_{j+1},\pi_{j+2},\ldots,\pi_{N}$.
\item Let $D$ be and irreducible component $D$ of $E$ and let $E'\subset E$ be the normal crossings divisor in $M$ obtained by eliminating $D$. Then $\mathcal M'=( M,E',\mathcal F)$ is also an almost radial foliated space and the sequence $\pi_{1},\pi_{2},\ldots,\pi_{N}$ of blowing-ups induces a resolution sequence for $\mathcal M'$.
\item If a blowing-up $\pi_j$ has a center $Y_j$ of codimension one in $M$, we can eliminate it. Proceeding in this way, we can obtain a resolution sequence such that all the centers have codimension at least two. Such resolution sequences will be called {\em adjusted resolution sequences}.
\item As an application of Proposition \ref{prop:indestructible}, if $Y_{j-1}$ is the center of $\pi_j$ and it has codimension at least two in $M$, then we necessarily have that $Y_{j-1}\subset \operatorname{Sing}(\mathcal F_{j-1})$.
\end{itemize}
\begin{remark} \label{radialregular}
If $\mathcal M$ has no singular points, then $\mathcal M$ is almost radial if and only if all the points in $M$ are simple for $\mathcal M$. In this case, the only adjusted resolution sequence is the empty one.
\end{remark}
Let $\mathcal S$ be a resolution sequence for an almost radial foliated space $$\mathcal M=(M,E,\mathcal F)$$ and consider a point $P\in M$. We can localize the resolution sequence at $P$ to obtain a resolution sequence $\mathcal S_P$ of the germ
$$
\mathcal M_P=((M,P),(E,P),\mathcal F_P)
$$
as follows. We consider the foliated spaces $\widetilde{\mathcal M}_j$ obtained by germifying $\mathcal M_j$ at the compact sets
$$
K_{j,P}=(\pi_1\circ\pi_2\circ\cdots\circ\pi_{j})^{-1}(P)
$$
and the transformations $\pi_{j,P}:\widetilde{\mathcal M}_j\rightarrow
\widetilde{\mathcal M}_{j-1}$ given by the restriction of $\pi_j$. Now, we skip all the $\pi_{j,P}$ that are the identity transformation, because the center $Y_{j-1}$ does not intersect the compact set $K_{j-1,P}$. As a consequence ot this procedure, we see that
$\mathcal M_P$ is also an almost radial foliated space.

In order to define {\em radial foliated spaces}, we need to precise the kind of resolution sequences we allow.
Consider a blowing-up morphism of ambient spaces
$$
\pi: (M',E')\rightarrow ((\mathbb C^n,\mathbf 0), E)
$$
with center $Y\subset ({\mathbb C}^n,\mathbf 0)$.
We say that $\pi$ is {\em $E$-controlled} if the union of the ideal of $Y$ with the ideals of the irreducible components of $E$ fulfill the maximal ideal of $\mathcal O_{\mathbb C^n,\mathbf 0}$.
\begin{remark}
If $Y=\{\mathbf 0\}$, the morphism $\pi$ is $E$-controlled, even when $E=\emptyset$. In the case that $Y=(x_1=x_2=\cdots=x_{n-1}=0)$ is a curve, the morphism $\pi$ is $E$-controlled if and only if there is an irreducible component of $E$ transverse to $Y$.
\end{remark}
\begin{definition} \label{def:radialfoliatedspace}
An almost radial foliated space $\mathcal M=(M,E,\mathcal F)$ is a {\em radial foliated space} if and only if there is a resolution sequence $\mathcal S$ for $\mathcal M$ such that the first blowing-up in $\mathcal S_P$ is $E$-controlled, for any $P\in M$.
\end{definition}

The main statement in this work is the following result that characterizes germs of radial foliated spaces in dimension three.

\begin{theorem}\label{main}
Consider a foliated space $\mathcal M_0=((\mathbb C^3,\mathbf 0),E^0,\mathcal F_0)$. with non-empty singular locus. Then $\mathcal M_0$ is a radial foliated space if and only if there are coordinates $x,y,z$ such that $\mathcal F_0$ is the ``open book'' foliation given by
$$
\omega=ydz-zdy.
$$
and $E^0\subset (xyz=0)$.
\end{theorem}

\section{Radial Foliated Spaces in Dimension Two} \label{Radial Foliated Spaces in Dimension Two}
In this Section \ref{Radial Foliated Spaces in Dimension Two}, we characterize germs of radial foliated spaces in dimension two, radial foliated spaces over the projective plane and radial foliated spaces over Hirzebruch surfaces.

In dimension two, the concepts of almost radial and radial foliated spaces are coincident as we see in the following proposition:

\begin{proposition} Let $\mathcal M=(M,E,\mathcal F)$ be a two dimensional foliated space. The following statements are equivalent:
\begin{enumerate}
\item[a)] $\mathcal M$ is an almost radial foliated space.
\item[b)] $\mathcal M$ is a radial foliated space.
\end{enumerate}
\end{proposition}
\begin{proof}
It is enough to note that the blowing-ups centered at points are automatically controlled.
\end{proof}

\subsection{Germs of Radial Foliated Spaces in Dimension Two}
\label{Germs of Radial Foliated Spaces in Dimension Two}
We say that a foliation $\mathcal F$ over $(\mathbb C^2,\mathbf 0)$ is a {\em cart-wheel foliation} when it is generated by the $1$-form $ydx-xdy$ in appropriate coordinates $x,y$.
\begin{remark}
We use the terminology ``cart-wheel foliation'' instead of ``radial foliation'' to avoid confusion with the definition of radial foliated space.
\end{remark}
In this Subsection \ref{Germs of Radial Foliated Spaces in Dimension Two}, we prove the following result:
\begin{theorem} \label{teo:dimensiontworadial}
A foliated space $((\mathbb C^2,\mathbf 0),E,\mathcal F)$ is a  radial foliated space over $(\mathbb C^2,\mathbf 0)$ if we are in one of the following situations:
\begin{enumerate}
\item[a)] The origin is a regular simple point.
\item[b)] The foliation $\mathcal F$ is a cart-wheel foliation and  $E=E_{\operatorname{inv}}$.
\end{enumerate}
\end{theorem}
Note that any foliated space $((\mathbb C^2,\mathbf 0),E,\mathcal F)$, where $E=E_{\operatorname{inv}}$ and $\mathcal F$ is a cart-wheel foliation, is a radial foliated space. To see this, we just have to perform a single blowing-up. If the origin is a regular point, and the foliated space is radial, we have Situation a), since the only possible adjusted reduction sequence is the empty one in view of Remark \ref{radialregular}. Thus, we have only to show that if the origin is singular and we have a radial foliated space, then we are in Situation b).

The end of the proof of Theorem \ref{teo:dimensiontworadial} will follow from Lemma \ref{lema:oneblowing-up1}:

\begin{lema}
 \label{lema:oneblowing-up1}
 Consider a foliated space $\mathcal M=\left((\mathbb C^2,\mathbf 0),\emptyset,\mathcal F\right)$ where the origin  $\mathbf 0$ is a singular point for $\mathcal F$ and let
$$
\pi:\mathcal M'=(M',E',\mathcal F')\rightarrow\mathcal M
$$
be the composition of a non-empty sequence of blowing-ups centered at singular points. Assume that the first blowing-up is non-dicritical or that the number of blowing-ups is bigger than or equal to two. Then, there is at least one point $P'$ in $M'$ that it is either a singular point  or a regular non-simple point for $\mathcal M'$.
\end{lema}
\begin{proof}
Assume first that all the blowing-ups in $\pi$ are dicritical ones. Since there are at least two blowing-ups, we find a point $P'\in E'$ that is the intersection of two irreducible components of $E'=E'_{\operatorname{dic}}$. This point $P'$ cannot be a simple regular point for $\mathcal M'$.

Let us consider the case when there is at least one non-dicritical blowing-up. After the first non-dicritical blowing-up, we know that the sum of Camacho-Sad indices with respect to the exceptional divisor is equal to $-1$. Hence, there is at least one index-persistent point. In view of Proposition \ref{prop:camachosad}, we get a singular point for $\mathcal M'$.
\end{proof}

\begin{proof}[End of the proof of Theorem \ref{teo:dimensiontworadial}]
We assume that the origin is singular for  $\mathcal F$ and that $\mathcal M$ is a radial foliated space. By eliminating components of the divisor, we know that  $(({\mathbb C^2},\mathbf 0),\emptyset,\mathcal F)$ is also a radial foliated space, in view of the statements in Subsection \ref{Radial Foliated Spaces}. By Lemma  \ref{lema:oneblowing-up1}, we know that the single blowing-up of the origin
$$
\pi:(M',E')\rightarrow (({\mathbb C}^2,\mathbf 0),\emptyset)
$$
gives the only resolution sequence for $(({\mathbb C^2},\mathbf 0),\emptyset,\mathcal F)$. The blowing-up $\pi$ must be dicritical and the transformed foliated space has only simple regular points. The transformed foliation is transverse to the exceptional divisor in all the points. This implies that $\mathcal F$ has order equal to one and the initial part corresponds to a cart-wheel foliation; hence it is a cart-wheel foliation in view of classical results of linearization. For more details, see \cite{Can-C-D}.

Let us show finally that $E=E_{\operatorname{inv}}$. Indeed, if $E_{\operatorname{dic}}$ is not empty, the dicritical blowing-up of the origin is not a resolution sequence. This ends the proof of Theorem \ref{teo:dimensiontworadial}.
\end{proof}
Let us recall that in dimension two the centers in the adjusted resolution sequences are always isolated points. Moreover, single points have normal crossings with any normal crossings divisor. This allows us to state the following general result:

\begin{corollary}
\label{cor:anhadircomponentes}
Let $\mathcal M=(M,E,\mathcal F)$ be a two dimensional foliated space, where $M$ is compact or a germ over a compact set. The following properties are equivalent:
\begin{enumerate}
\item $\mathcal M$ is a radial foliated space.
\item The foliated space $(M,\emptyset,\mathcal F)$ is radial and the normal crossings divisor $E$ satisfies the next properties:
    \begin{enumerate}[a)]
    \item If $D$ is a component of $E$ such that $D\cap \operatorname{Sing}(\mathcal F)\ne\emptyset$, then $D$ is non-dicritical.
    \item Dicritical components are transverse to $\mathcal F$ at
    each point.
    \item Two dicritical components of $E$ are mutually disjoint.
    \end{enumerate}
\end{enumerate}
\end{corollary}
\begin{proof}
Compactness allows us to work in a local way. Now the result follows from Theorem \ref{teo:dimensiontworadial}.
\end{proof}

\subsection{Radial Foliations Over the Projective Plane}
\label{Radial Foliations Over the Projective Plane}
Let us recall that the degree $d$ of a foliation $\mathcal F$ over the projective plane $\mathbb P^2_{\mathbb C}$ is the number of tangencies with a generic projective line (see \cite{Can-C-D}). In homogeneous coordinates $X_0,X_1,X_2$, such a foliation is given by an homogeneous differential form
$$
W=\sum_{i=0}^2A_i(X_0,X_1,X_2)dX_i,\quad \sum_{i=0}^2X_iA_i=0,
$$
where the coefficients $A_i$ are homogeneous polynomials of degree $d+1$, without common factor.

In particular, a foliation of degree zero is given by the projective lines passing through a common point $P$ which is the only singular point of the foliation.

In a more general setting, the number of singular points of a degree $d$ foliation over $\mathbb P^2_{\mathbb C}$, counted with the multiplicity given by the Milnor number, is equal to $d^2+d+1$. The reader can see \cite{Can-C-D} for an elementary proof of this fact.

This Subsection \ref{Radial Foliations Over the Projective Plane}  is devoted to prove the following statement:
\begin{proposition} \label{prop:radialinprojectiveplanes}
Let $\mathcal M=({\mathbb P^2_{\mathbb C}, E,\mathcal F})$ be a foliated space over the projective plane $\mathbb P^2_{\mathbb C}$. The following properties are equivalent:
\begin{enumerate}
\item $\mathcal M$ is a radial foliated space.
\item The foliation $\mathcal F$ has degree $0$. Moreover, the divisor $E_{\operatorname{inv}}$ is a union of at most two projective lines through the singular point $P$; on the other hand, the divisor $E_{\operatorname{dic}}$ is either empty or just one projective line not containing $P$.
\end{enumerate}
\end{proposition}
\begin{proof}
The fact that (2) implies (1) follows from the arguments in Subsection \ref{Germs of Radial Foliated Spaces in Dimension Two}. Let us show that (1) implies (2). Let us first show that $\mathcal F$ has degree $d=0$.
	
There is an homological invariant called Baum-Bott Index, associated to each point $P$ and a foliation $\mathcal F$ over a surface. It allows to compute the self-intersection number of the normal fiber bundle $\mathcal N_\mathcal F$ of the foliation as follows
\begin{equation} \label{eq:baumbott}
\mathcal N_\mathcal F\cdot\mathcal N_\mathcal F=\sum_{P\in \operatorname{Sing} \mathcal F}\operatorname{BB}(\mathcal F,P).
\end{equation}
(For details, see \cite{Bru}). On the other hand, we know that $\mathcal N_\mathcal F=\mathcal O(2+d)$ and hence the self-intersection is given by
$$
\mathcal N_\mathcal F\cdot\mathcal N_\mathcal F=(2+d)^2=4+4d+d^2.
$$

The Baum-Bott Index can be computed in the case of singular points with a non-nilpotent linear part as being the quotient between the square of the trace and the determinant of the linear part. In the case of a cart-wheel singular point we have that
$$
\operatorname{BB}(\mathcal F,P)=4.
$$

Note that we are dealing with a foliation $\mathcal F$ of degree $d$ such that all the singularities are of cart-wheel type. The Milnor number of a cart-wheel foliation is exactly equal to one. So $k=1+d+d^2$ counts the number of singular points in this case.

Then, Baum-Bott formula in Equation \eqref{eq:baumbott} says that
$$
4+4d+d^2=4k=4(1+d+d^2).
$$
Hence $d^2=4d^2$, that is  $d=0$.

Now, since the divisor $E$ must be non-dicritical through the only singular point $P$, it is the union of at most two lines through $P$ and some irreducible components not containing $P$.  If $D$ is a component of the divisor not containing $P$, it is necessarily dicritical, since the leaves are the lines through $P$. We necessarily have that $D$ is a projective line, otherwise, we find tangents with $\mathcal F$. There is at most one of that dicritical components, since two of them cannot meet (see Corollary \ref{cor:anhadircomponentes}).
\end{proof}

\subsection{Foliations over Hirzebruch Surfaces} We recall some known facts on Hirzebruch surfaces and their foliations. The reader can find more details in \cite{Gal-M-O}.

Recall that Hirzebruch surfaces are the fiber bundles over $\mathbb P^1_{\mathbb C}$ having fiber equal to $\mathbb P^1_{\mathbb C}$. They are classified by an index $\delta\in \mathbb Z_{\geq 0}$, each one being isomorphic to a surface $S_{\delta}$, that we describe below.

The Hirzebruch surface $S_{\delta}$ is given by an atlas
$$
\mathcal A_\delta=\{(U_{ij}, (x_{ij},y_{ij}))\}_{0\leq i,j\leq 1},
$$
where the four coordinate sets $U_{ij}$ are identified to $\mathbb C^2$, the coordinate changes are given by the following equations
\begin{equation} \label{eq:cambioHirzebruch}
\begin{array}{l}
x_{00}=x_{01}=1/x_{10}=1/x_{11},\\
y_{00}=1/y_{01}=x_{10}^\delta y_{10}=x_{11}^\delta / y_{11},
\end{array}
\end{equation}
and the intersections $U_{ij}\cap U_{k\ell}$ are defined by
$$
\begin{array}{lclcl}
	U_{00}\cap U_{01}&=&(y_{00}\ne0)&=&(y_{01}\ne0),\\
	U_{00}\cap U_{10}&=&(x_{00}\ne0)&=&(x_{01}\ne0),\\
	U_{00}\cap U_{11}&=&(x_{00}y_{00}\ne0)&=&(x_{11}y_{11}\ne0),\\
	U_{01}\cap U_{10}&=&(x_{01}y_{01}\ne0)&=&(x_{10}y_{10}\ne0),\\
	U_{01}\cap U_{11}&=&(x_{01}\ne0)&=&(x_{11}\ne0),\\
	U_{10}\cap U_{11}&=&(y_{10}\ne0)&=&(y_{11}\ne0).
\end{array}
$$

Let us describe now the field $\mathbb C(S_\delta)$ of rational funcions over $S_\delta$. Consider the polynomial ring $\mathbb C[X_0,X_1,Y_0,Y_1]$ as a bi-graded algebra where the bi-degrees of the variables are given as follows:
$$
X_0\mapsto (1,0),\; X_1\mapsto (1,0),\;Y_0\mapsto (0, 1),\; Y_1\mapsto (-\delta,1).
$$
This gives rise to the notion of {\em bi-homogeneous polynomial of bidegree (a,b)}. In this setting, the field of rational functions $\mathbb C(S_\delta)$ is given by the quotients of bi-homogeneous polynomials of same bi-degree. The rational functions $x_{ij}$ and $y_{ij}$  correspond to the following quotients:
\begin{eqnarray*}
x_{00}=X_1/X_0 ,\; y_{00}=X_0^\delta  Y_1/Y_0,\\
x_{10}= X_0/X_1 ,\; y_{10}=X_1^\delta Y_1/Y_0,\\
x_{01}= X_1/X_0,\; y_{01}=Y_0/ X_0^\delta  Y_1,\\
x_{11}= X_0/X_1 ,\; y_{11}= Y_0/ X_1^\delta  Y_1.
\end{eqnarray*}

The group of classes of divisors is generated by $F=(X_0=0)$ and $L=(Y_0=0)$ with the following intersection pairing
$$
F\cdot F=0,\quad F\cdot L=1,\quad L\cdot L=\delta.
$$
The curve $L_0=(Y_1=0)$ is the only curve with negative self-intersection $L_0\cdot L_0=-\delta$. For $d_1,d_2\geq 0$, the class $d_1F+d_2L$ is represented by $P=0$, where $P$ is a bi-homogeneous polynomial of bi-degree $(d_1,d_2)$.

Note that $L_0$ is covered by the charts $U_{00}$ and $U_{10}$. The corresponding changes of coordinates in this charts are
\begin{equation} \label{eq:cambiocartasL0}
x_{10}=1/x_{00}, \; y_{10}=x_{00}^{\delta}y_{00}.
\end{equation}

Let us give a practical way of describing the singular foliations over $S_\delta$. Let $\mathcal F$ be a foliation over $S_\delta$. Then, the tangent bundle $T_{\mathcal F}$ is linearly equivalent to $-d_1F-d_2L$ and we call $(d_1,d_2)$ the {\em bi-degree of $\mathcal F$}. We know that the normal bundle $N_\mathcal F$ is linearly equivalent to $aF+bL$, where
$$
a=d_1+2-\delta,\quad  b=d_2+2.
$$
The bi-degree is also defined associated to a differential monomial, by assigning the same bi-degree to $X_i$ and $dX_i$ and to $Y_j$ and $dY_j$. In this way, a foliation $\mathcal F$ such that the normal bundle $\mathcal N_\mathcal F$ is linearly equivalent to $aF+bL$ is defined by a bi-homogeneous differential form
\begin{equation}
\label{eq:bihomogeousform}
W=A_0dX_1+A_1dX_1+B_0dY_0+B_1dY_1
\end{equation}
of bi-degree $(a,b)$, such that the following Euler type equations hold:
\begin{equation*}
X_0A_0+X_1A_1-\delta Y_1B_1=0;\quad Y_0B_0+Y_1B_1=0.
\end{equation*}
(The reader should be able to give generators of $\mathcal F$ in the affine charts, just by a deshomogeneization of $W$).
\begin{lema} Let $\mathcal F$ be a foliation over $S_\delta$ of bi-degree $(d_1,d_2)$. Let $m(d_1,d_2)$ be the number of singular points of $\mathcal F$ counted with the multiplicity given by the Milnor number. Then, we have that
$$
m(d_1,d_2)= (d_2+1)[2(d_1+1)+\delta d_2]+2.
$$
\end{lema}
\begin{proof} It can be deduced by taking a specific foliation of bi-degree  $(d_1,d_2)$. It can also be obtained in a more intrinsic way as follows (see \cite{Bru}). We know that
$$
m(d_1,d_2)=c_2(S_\delta)-T_\mathcal F \cdot N_\mathcal F,
$$
where $c_2(S_\delta)$ is the second Chern class of $S_\delta$, which is known to be equal to  $4$. We obtain that
$
m(d_1,d_2)=(d_2+1)[2(d_1+1)+\delta d_2]+2$,
as desired.
\end{proof}
\subsection{Radial Foliations over Hirzebruch Surfaces}
\label{Radial Foliations over Hirzebruch Surfaces}
In this subsection we characterize radial foliated spaces over Hirzebruch surfaces.
\begin{proposition} 
\label{prop: radialesHirzebruch}
Let $\mathcal M=(S_\delta,\emptyset,\mathcal F)$ be a radial foliated space. We have that:
\begin{enumerate}
\item If $\delta\ne 0$, then $\mathcal F$ coincides with the only projective fibration defining $S_\delta$.
\item If $\delta=0$, then $\mathcal F$ is one of the  two possible projective fibrations defining $S_0=\mathbb P^1_\mathbb C\times\mathbb P^1_\mathbb C$.
\end{enumerate}
Note that, in any case $\operatorname{Sing}(\mathcal F)=\emptyset$.
\end{proposition}

The proof of the above Proposition \ref{prop: radialesHirzebruch} is based on the Baum-Bott formula introduced in Equation \eqref{eq:baumbott}.

Let $\mathcal F$ be a foliation over $S_\delta$ of bi-degree $(d_1,d_2)$ defined by a bi-homogeneous $1$-form $W$ as in Equation
\eqref{eq:bihomogeousform}
of bi-degree $(a,b)$, where
$$
a=d_1+2-\delta, \quad b=d_2+2.
$$
We assume that the foliated space $(S_\delta,\emptyset,\mathcal F)$ is a radial foliated space, that is, all the singularities are cart-wheel singularities. Since the singularities have Milnor number equal to one, we have exactly
$$
m(d_1,d_2)=(d_2+1)[2(d_1+1)+\delta d_2]+2
$$
singular points. Moreover, the Baum-Bott index of each singularity is equal to $4$. Hence the Baum-Bott formula stands as
$$
\mathcal N_\mathcal F\cdot\mathcal N_\mathcal F=2(d_1+2-\delta)(d_2+2)+\delta (d_2+2)^2= 4 m(d_1,d_2).
$$
We obtain the equation
\begin{equation}\label{eq:equationbaumbott}
6d_1d_2+4d_1+4d_2+3\delta d_2^2+2\delta d_2+8=0
\end{equation}
Let us show that Equation \eqref{eq:equationbaumbott} leads us to prove Proposition \ref{prop: radialesHirzebruch}.

If we isolate $d_1$, we get
$$
d_1=\frac{-(3\delta d_2^2+d_2(4+2\delta)+8)}{6d_2+4}=-\left(\frac{\delta d_2}{2}+1\right)+\frac{d_2-2}{3d_2+2}.
$$
From here, we distinguish two cases:
\begin{itemize}
\item Either $\delta$ or $d_2$ is even. Since $d_1\in \mathbb Z$, it must happen that $$(d_2-2)/(3d_2+2)\in \mathbb Z.$$ This only occurs if $d_2 \in \{-2,-1,0,2\}$.
\item Both $\delta$ and $d_2$ are odd. Since $d_1\in \mathbb Z$ we must have that $$(5d_2-2)/(6d_2+4)\in \mathbb Z.$$  This never happens.
\end{itemize}

\newpage
As a consequence, we need to study the following four situations:
\begin{enumerate}
	\item $d_2=-2$. In this case we have $d_1=\delta$ and $m(d_1,d_2)=0$.
	\item $d_2=0$. In this case we have $d_1=-2$ and $m(d_1,d_2)=0$.
 	\item $d_2=2$. In this case we have $d_1=-\delta-1$ and $m(d_1,d_2)=2$.
 	\item $\delta \in 2\mathbb Z$ and $d_2=-1$. In this case we would have $d_1=2+\delta/2$ and $m(d_1,d_2)=2$.
\end{enumerate}
Let us consider one by one these situations:
\begin{itemize}
\item{\em Situation 1 holds when $\mathcal F$ is the foliation given by the fibers.}
The bi-degree of $W$ must be $(2,0)$, this implies $B_0=B_1=0$ and the foliation is given by  $W= X_1dX_0-X_0dX_1$.

\item{\em Situation 2 holds if and only if $\delta=0$.} Assume that $\delta >0$. We are asking $B_0$ and $B_1$ to have bi-degrees $(-\delta,1)$ and $(0,1)$, respectively. We have necessarily that $B_0=cY_1$, with $c \in \mathbb C^*$ and thanks to the Euler-type condition $Y_0B_0+Y_1B_1=0$, we get $B_1=-cY_0$. This is incompatible with the equation
$$
X_0A_0+X_1A_1=\delta Y_1(-cY_0).
$$
When $\delta=0$, we are like in situation 1, but we obtain ``the other fibration'' given by $W=Y_1dY_0-Y_0dY_1$.

\item {\em Situation 3 never holds.} We are asking $A_0$ and $A_1$ to have bi-degree $(-2\delta,4)$. Assume first that $\delta >0$. We get necessarily that $A_i=Y_1^2\tilde A_i$ for $i=0,1$, with $\tilde A_i$ of bi-degree $(0,2)$. From the equation
$$
X_0A_0+X_1A_1=\delta Y_1B_1,
$$
we get that also $Y_1$ divides $B_1$. On the other hand, by the equation $Y_0B_0+Y_1B_1=0$ also $Y_1$ divides $B_0$ and this contradicts the fact that $W$ is without common factors.

When $\delta=0$ the conditions $X_0A_0+X_1A_1=0$ and $A_0, A_1$ having bi-degree $(0,4)$ together implies that  $A_0=A_1=0$. Condition $Y_0B_0+Y_1B_1=0$ and the fact that there are no common factors implies now that
$
W= Y_1dY_0-Y_0dY_0
$,
which implies $d_1=-2,d_2=0$. Contradiction.
\item{\em Situation 4 never holds.} We are asking $B_0$ and $B_1$ to have bi-degrees $(4-\delta/2,0)$ and $(4+\delta/2,0)$, respectively. This is compatible with the condition $Y_0B_0+Y_1B_1=0$ only if $B_0=0$ and $B_1=0$. Since there are no common factors, we get necessarily that $W= X_1dX_0-X_0dX_1$, which implies $d_1=\delta$ and $d_2=-2$. Contradiction.
\end{itemize}

\begin{corollary}
Let $(S_\delta,E,\mathcal F)$ be a radial foliated space over the Hirzebruch surface $S_\delta$, with $\delta \ne 0$. The foliation $\mathcal F$ is given by the fibers of the fibration $S_\delta \to \mathbb P^1_\mathbb C$. Moreover, we have that:
\begin{enumerate}
	\item The non-dicritical divisor $E_{\operatorname{inv}}$ is a finite union of invariant fibers.
	\item  The dicritical divisor $E_{\operatorname{dic}}$ is the union of at most two rational curves transversal to the fibers. If we have two components, they are not linearly equivalent.
\end{enumerate}
\end{corollary}
\begin{proof}
A curve transversal to the fibers is always rational and a generator of the Picard group jointly with a fiber. Two of them are either linearly equivalent or not. In the first case, they meet at $\delta$ points. In the second case one of them is the curve having negative self-intersection and the two curves do not intersect.	Now, the result follows from Corollary \ref{cor:anhadircomponentes}.
\end{proof}

\section{Monoidal Blowing-ups}
In this section we introduce some useful features of blowing-ups centered at curves in ambient spaces of dimension three.
\subsection{Hirzebruch Tubes}
Let us consider an ambient space $(M,E)$ of dimension three.  Given an irreducible non-singular curve  $Y\subset M$, we say that $Y$ has {\em full normal crossings with $E$} if the following properties hold:
\begin{enumerate}
\item The curve $Y$ has normal crossings with $E$.
\item The curve $Y$ is contained in each irreducible component $D$ of $E$ such that $Y\cap D\ne\emptyset$.
\end{enumerate}

\begin{definition}
Let us consider the blowing-up
$$
\pi:(M',E')\rightarrow (M,E),
$$
with center a curve $Y$ having full normal crossings with $E$. A non-singular curve $Y'\subset M'$ is said to be a {\em $1$-infinitely near curve of $Y$} if $\pi$ defines an étale surjective morphism $Y'\rightarrow Y$ (that is, de morphism $Y' \to Y$ is a non-ramified covering) and $Y'$ has full normal crossings with $E'$.
\end{definition}

\begin{remark}
In the above situation, we have a ``vertical foliation'' on the divisor $\pi^{-1}(Y)$ given by the fibers $\pi^{-1}(P)$ over the points $P \in Y$. A $1$-infinitely near curve $Y'$ of $Y$ is contained in $\pi^{-1}(Y)$ and it is transverse to that foliation. The converse is also true. Namely, if we have an irreducible  non-singular curve $Y'\subset \pi^{-1}(Y)$ such that $Y'$ has full normal crossings with $E'$ and it is transverse to the vertical foliation, then $Y'$ is a $1$-infinitely near curve of $Y$.
\end{remark}

Let us define now the concept of {\em Hirzebruch tube $(M,E;Y)$ of order $(\alpha,\beta)$}, for $\alpha\in \mathbb Z_{\geq 0}$ and $\beta\in \mathbb Z_{\geq 1}$:
\newpage
\begin{definition} 
A {\em Hirzebruch tube of order $(\alpha,\beta)$}, with $\alpha\in \mathbb Z_{\geq 0}$ and $\beta\in \mathbb Z_{\geq 1}$, is a triple $(M,E;Y)$, where $(M,E)$ is a three-dimensional ambient space and $Y\subset E$ is an irreducible non-singular curve having full normal crossings with $E$. We ask the space $M$ to be covered by two charts $U_1,U_2$ isomorphic to $\mathbb C\times(\mathbb C^2,\mathbf 0)$, with respective coordinate functions $(x,y,z)$ and $(u,v,w)$ satisfying the following properties:
\begin{enumerate}
\item The coordinate changes are given by
$$
u=1/x,\quad  v=x^\beta y,\quad w=z/x^\alpha.
$$

\item The divisor $E$ has one or two irreducible components. If $E$ has a single irreducible component $E_1$, then
$$
E_1\cap U_1=(y=0), \quad E_1\cap U_2=(v=0).
$$
If $E$ has two irreducible components $E_1$ and $E_2$, then
$$
\begin{array}{l}
E_1\cap U_1=(y=0), \quad E_1\cap U_2=(v=0), \\
E_2\cap U_1=(z=0), \quad E_2\cap U_2=(w=0).
\end{array}
$$
The component $E_1$ will be called the {\em marked component of $E$}.
\item The curve $Y$ is given by the equations $Y\cap U_1=(y=z=0)$ and $Y\cap U_2=(v=w=0)$. Hence $Y$ is isomorphic to $\mathbb P^1_{\mathbb C}$ and the functions $x$, $u$ define the two standard affine charts of the projective line $Y$.
\end{enumerate}
\end{definition}
Let $(M,E)$ be a three dimensional ambient space and consider a compact irreducible curve $Y\subset M$. If $(\tilde M,\tilde E;Y)$ is a Hirzebruch tube, where $(\tilde M,\tilde E)$ denotes the germ of $(M,E)$ along $Y$, we also say that $(M,E;Y)$ is a {\em Hirzebruch tube}.

\begin{proposition}
\label{prop:alphabetaHirzebruch1}
Let $(M,E;Y)$ be a Hirzebruch tube of order $(\alpha,\beta)$. Consider the blowing-up $\pi:(M',E')\rightarrow (M,E)$ with center $Y$. Then we have the following properties:
\begin{enumerate}[a)]
\item The exceptional divisor $\pi^{-1}(Y)$ is isomorphic to the Hirzebruch surface $S_{\alpha+\beta}$. (Note that $\alpha+\beta\geq 1$).
\item Let $E'_1$ be the strict transform by $\pi$ of the marked component $E_1$ of $E$. The non-singular curve $L_0=E'_1 \cap \pi^{-1}(Y)$ is the unique irreducible curve with negative  self-intersection in $\pi^{-1}(Y)$.
\item Let $Y' \subset M'$ be a $1$-infinitely near curve of $Y$. Then, we have that $(M',E'; Y')$ is a  Hirzebruch tube.
\item The only non-singular foliation on $\pi^{-1}(Y)$ is given by the fibers of $\pi$ over the points of $Y$.
\end{enumerate}
\end{proposition}
\begin{proof}
Since $Y$ is a projective line and $\pi$ induces a fibration on $\pi^{-1}(Y)$ with fibers isomorphic to $\mathbb P^1_{\mathbb C}$, we deduce that $\pi^{-1}(Y)$ is a Hirzebruch surface and the fibers define a non-singular foliation on it. If we show that the index of the Hirzebruch surface is greater than or equal to $1$, the statement d) holds as a consequence of Proposition \ref{prop: radialesHirzebruch}. Then, let us see that the index of the Hirzebruch surface $\pi^{-1}(Y)$ is precisely $\alpha+\beta\geq 1$ and that $L_0$ is the unique irreducible curve with self-intersection $-(\alpha+\beta)$.

The space $M'$ is covered by four charts $U_{11},U_{21},U_{12},U_{22}$ with res\-pective coordinates $(x_1,y_1,z_1)$, $(u_1,v_1,w_1)$, $(x_2,y_2,z_2)$ and $(u_2,v_2,w_2)$. The blowing-up morphism $\pi$ is given by the equations
\begin{equation*} \label{eq:blow-up}
\begin{array}{llllllllllll}
	x&=&x_1&=&x_2 &&& u&=&u_1&=&u_2 \\
	y&=&y_1z_1&=&y_2 &&& v&=&v_1w_1&=&v_2 \\
	z&=&z_1&=&y_2z_2 &&& w&=&w_1&=&v_2w_2
\end{array}
\end{equation*}
In these coordinates we have that $E'_1=(y_1=0)=(v_1=0)$ and
$$
\pi^{-1}(Y)=(z_1=0)=(y_2=0)=(w_1=0)=(v_2=0).
$$
The change of coordinates concerning the curve $L_0=E'_1 \cap \pi^{-1}(Y)$ is given by
$$
u_1=1/x_1, \quad v_1=x_1^{\alpha+\beta}y_1, \quad w_1=z_1/x^{\alpha}.
$$
With the notations in Equation \eqref{eq:cambiocartasL0}, we have that $x_1,y_1,u_1,v_1$ play the role of $x_{00},y_{00},x_{10},y_{10}$, respectively. This ends the proof of statements $a),b)$ and $d)$.

Let us prove statement c). When the curve $Y'$ is the special curve $L_0$, we are done by taking $\alpha'=\alpha$ and $\beta'=\alpha+\beta$, in view of the previous computations.

Assume now that the curve $Y'$ is different from $L_0$. A first remark is that $Y' \cap L_0=\emptyset$, hence, we have that $Y' \subset U_{12} \cup U_{22}$. Moreover, we have that $Y'$ is a projective line. Indeed, the curve $Y'$ is a non-ramified covering of the projective line $Y$ through the blowing-up; hence the covering is $1$ to $1$ and it is a projective line.

Let us see that $Y'$ is a curve in $\pi^{-1}(Y)$ of bi-degree $(0,1)$. Let $F$ be the linear equivalence class in $\pi^{-1}(Y)=S_{\alpha+\beta}$ of a fiber of the blowing-up. Moreover, let $L$ be the equivalence class of divisors such that $L\cdot L_0=0$ and $L\cdot F=1$ and $L\cdot L=\alpha+\beta$. We know that $F$ and $L$ generate the Picard group of $S_{\alpha+\beta}$ and then we can write $[Y']=aF+bL$, where $(a,b)$ is the bi-degree of $Y'$. Let us compute $(a,b)$.  We have seen that
$$
F\cdot [Y']=1.
$$
This implies that $F\cdot [Y']=b=1$. On the other hand, we know that $L_0\cdot Y'=0$, this implies that $$
0=L_0\cdot Y'=a+bL_0\cdot L=a.
$$
Hence, the bidegree of $Y'$ is $(0,1)$.  As a consequence, the curve $Y'$ is defined by $Y' \cap U_{12}=(\tilde z_2=0)$ and $Y' \cap U_{22}=(\tilde w_2=0)$, where
$$
\tilde z_2=z_2+\sum_{i=0}^{\alpha+\beta}a_i x_2^{\alpha+\beta-i}=0, \quad
\tilde w_2=w_2+\sum_{i=0}^{\alpha + \beta} a_i u_2^i=0.
$$
and the coordinates $(x_2,y_2,\tilde z_2)$, $(u_2,v_2,\tilde w_2)$ defined in $U_{12}$ and $U_{22}$, res\-pectively, give us the desired Hirzebruch tube around $Y'$ of degree $(\alpha',\beta')$ with $\alpha'=\alpha+\beta$ and $\beta'=\beta$.
\end{proof}
\begin{remark}
If $Y'$ is different from $L_0$ and $E$ has two components, then $Y'=\pi^{-1}(Y) \cap E'_2$, where $E'_2$ is the strict transform by $\pi$ of the non-marked component $E_2$; this is because $Y'$ is a $1$-infinitely near curve of $Y$. Both in the cases when $E$ has one or two components, the marked component around $Y'$ is given by the exceptional divisor $\pi^{-1}(Y)$. In the case when $Y'=L_0$, the new marked component is the strict transform $E'_1$ of the marked component $E_1$.
\end{remark}

\subsection{Vertical Blowing-ups}
\label{Vertical Blowing-ups}

This subsection is devoted to characterize the types of dicritical monoidal blowing-ups in terms of equations. We start with the definition of {\em vertical blowing-up}:

\begin{definition}
 \label{def:vertical}
   Consider a blowing-up $\pi:(M',E',\mathcal F)\rightarrow (M,E,\mathcal F)$ of ambient spaces with center an irreducible non-singular curve $Y\subset M$. We say that $\pi$ is {\em vertical} if and only if it is dicritical and the foliation induced by $\mathcal F'$ on $\pi^{-1}(Y)$ is the one given by the fibers $\pi^{-1}(P)$, when $P$ varies over the points of $Y$.
\end{definition}
\begin{remark} In the above definition we do not ask $Y$ to be invariant for $\mathcal F$, just to have normal crossings with $E$. On the other hand, the role of the exceptional divisor $E$ in the definition is only relevant in what it concerns with the normal crossings condition. Namely, the blowing-up $\pi$ is vertical if and only if the induced blowing-up
$$
(M',\pi^{-1}(Y),\mathcal F')\rightarrow (M,\emptyset,\mathcal F)
$$
is vertical. Moreover, since $Y$ is supposed to be connected, the condition of being vertical may be tested after localizing the blowing-up $\pi$ at a given point $P\in Y$.
\end{remark}

Consider a foliated space $\mathcal M=(({\mathbb C^n},\mathbf 0),\emptyset,\mathcal F)$. Take local coordinates $\mathbf z=(x,\mathbf y)$, with $\mathbf y=(y_2,y_3,\ldots,y_n)$ and consider the non-singular curve $Y$ given by
$$
Y=(y_2=y_3=\cdots=y_n=0).
$$
Assume that the foliation $\mathcal F$ is generated by the $1$-form
$$
\omega=a_1(\mathbf z)dx+\sum_{i=2}^na_i(\mathbf z)dy_i=a_1(\mathbf z)dx+\sum_{i=2}^ny_ia_i(\mathbf z)\frac{dy_i}{y_i},
$$
where the coefficients $\{a_i(\mathbf z)\}_{i=1}^n$ are without common factor.

Let
$
\pi:\mathcal M'=(M',D,\mathcal F')\rightarrow\mathcal M
$
be the blowing-up with center $Y$, where $D=\pi^{-1}(Y)$ is the exceptional divisor. Let us characterize under what conditions $\pi$ is a dicritical blowing-up and$/$or a vertical blowing-up.

Given a germ of holomorphic function $f(\mathbf z)$, we can write
$$
\textstyle
f(\mathbf z)=\sum_{s=0}^\infty F_s(x;\mathbf y),
$$
where each $F_{s}(x;\mathbf y)$ is a homogeneous polynomial of degree $s$ in the variables $\mathbf y$. The generic order $\nu_Y(f(\mathbf z))$ is defined to be
$$
\nu_Y(f(\mathbf z))=\min\{s\geq 0;\; F_s(x;\mathbf y)\ne 0 \}.
$$
When $f(\mathbf z)$ is identically zero, we put $\nu_Y(f(\mathbf z))=\infty$.

Denote by $r$ the ``log-generic order'' of $\omega$ along $Y$ given by
\begin{eqnarray*}
	r&=&\min\{\nu_Y(a_1(\mathbf z)), \nu_Y(y_2a_2(\mathbf z)), \ldots, \nu_Y(y_na_n(\mathbf z))\}= \\
	&=& \min\{\nu_Y(a_1(\mathbf z)), \nu_Y(a_2(\mathbf z))+1, \ldots, \nu_Y(a_n(\mathbf z))+1\} .
\end{eqnarray*}
Let us write
$
p(\mathbf z)=\sum_{i=2}^ny_ia_i(\mathbf z)
$.
There are two possibilities:
$$
\operatorname{\textbf{NDic:}}\;  \nu_Y(p(\mathbf z))=r, \qquad \operatorname{\textbf{Dic:}}\; \nu_Y(p(\mathbf z))\geq r+1.
$$
We divide the case \textbf{Dic} into two options:
\begin{description}
\item[Dic-v] $\nu_Y(y_ia_i(\mathbf z)) \geq r+1$ for all $i=2,3,\ldots,n$. Note that, in this case, we have that $\nu_Y(a_1(\mathbf z))=r$.
\item [Dic-nv] $\nu_Y(y_{i_0}a_{i_0}(\mathbf z))=r$, for some $i_0\in \{2,3,\ldots,n\}$.
\end{description}
Note that in the case {\bf Dic-nv} there are in fact at least two different indices $2 \leq j_0 < i_0 \leq n$ such that $\nu_Y(y_{j_0}a_{j_0}(\mathbf z))=\nu_Y(y_{i_0}a_{i_0}(\mathbf z))=r$, hence we can take $i_0 \geq 3$.
\begin{lema}
\label{lema:dicriticalnotdicritical}
We have the following equivalences:
\begin{enumerate}
\item $\pi$ is non-dicritical if and only if we are in case {\bf NDic}.
\item $\pi$ is dicritical if and only if we are in case {\bf Dic}.
\item  $\pi$ is vertical if and only if we are in case {\bf Dic-v}.
\item  $\pi$ is dicritical non-vertical if and only if we are in case {\bf Dic-nv}.
\end{enumerate}
\end{lema}
\begin{proof}
The four equivalences can be read in any of the $n-1$ standard charts of the blowing-up $\pi$. Take the chart that is given in equations by $\mathbf z'=(x,\mathbf y')$, where
	$$
	y_2=y'_2,\quad y_j=y'_2y'_j,\; j=3,4,\ldots,n.
	$$
	Recall that $y'_2=0$ is an equation of  the exceptional divisor $D=\pi^{-1}(Y)$ in this chart.
	The pull-back $\pi^*\omega$ is given by
	$$
	\pi^*\omega=a_1(\mathbf z)dx+\frac{p(\mathbf z)}{y'_2}dy'_2+y'_2\sum_{\ell=3}^ny'_\ell a_{\ell}(\mathbf z)dy'_\ell.
	$$
	If we are in the case ${\bf NDic}$, the maximum power of the exceptional divisor that divides $\pi^*\omega$ is ${y'_2}^{r-1}$ an thus the transformed foliation $\mathcal F'$ is locally generated by
	\begin{eqnarray*}
		\omega'&=&\frac{\pi^*\omega}{{y'_2}^{r-1}}
		=y_2'\frac{a_1(\mathbf z)}{{y'_2}^r}dx+
		\frac{p(\mathbf z)}{{y'_2}^{r}}dy'_2+
		y'_2\sum_{\ell=3}^ny'_\ell\frac{ a_{\ell}(\mathbf z)}{{y'_2}^{r-1}}dy'_\ell.
	\end{eqnarray*}
In this case, we have that $y'_2$ divides the coefficients of $\omega'$ for
$$
dx,dy_3',dy'_4,\ldots, dy_n'.
$$
Hence the exceptional divisor $(y'_2=0)$ is invariant and thus the blowing-up $\pi$ is non-dicritical.
	
	Assume now that we are in case {\bf Dic}. Then, the maximum power of the exceptional divisor that divides $\pi^*\omega$ is ${y'_2}^r$ an thus the transformed foliation $\mathcal F'$ is locally generated by
	\begin{eqnarray*}
		\omega'=\frac{\pi^*\omega}{{y'_2}^r}
		&=&\frac{a_1(\mathbf z)}{{y'_2}^r}dx+
		\frac{p(\mathbf z)}{{y'_2}^{r+1}}dy'_2+
		\sum_{\ell=3}^ny'_\ell\frac{ a_{\ell}(\mathbf z)}{{y'_2}^{r-1}}dy'_\ell
		\\
		&=&a'_1(\mathbf z')dx+\sum_{i=2}^n a'_i(\mathbf z')dy'_i.
	\end{eqnarray*}
 At least one of the coefficients $a'_1(\mathbf z'), a'_3(\mathbf z'),a'_4(\mathbf z'),\ldots,a'_n(\mathbf z')$ is not divisible by $y'_2$. Hence the blowing-up $\pi$ is dicritical.

 Assume now that we are in case {\bf Dic-nv}. We have that there is $i_0\in \{3,4,\ldots,n\}$ such that $y'_2$ does not divide  the coefficient  $a'_{i_0}({\mathbf z'})$.  For $\pi$ being vertical we need the foliation $\mathcal G$ to be defined by $dx=0$, thus $y'_2$ should divide $a'_3({\mathbf z'}),a'_4({\mathbf z'}),\ldots,a'_n({\mathbf z'})$, which is not happening.

If we are in the case {\bf Dic-v}, we have that $y'_2$ divides $a'_i({\mathbf z'})$, for all $i=3,4,\ldots,n$ and $y'_2$ does not divide $a'_1(\mathbf z')$.  The restriction of $\omega'$ to $\pi^{-1}(Y)$ is given by
	$$
	a'_1(x;0,y'_3,y'_4,\ldots,y'_n)dx=0.
	$$
Then, the foliation $\mathcal G=\mathcal F'|_D$ is exactly $dx=0$. That is, the blowing-up $\pi$ is vertical.
\end{proof}
\begin{proposition}
\label{prop:verticalblowingups}
	Let $
	\pi:\mathcal M'=(M',E',\mathcal F')\rightarrow \mathcal M=(M,E,\mathcal F)
	$ be a vertical blowing-up with center $Y$. Assume that all the points in $\pi^{-1}(Y)$ are simple regular points for $\mathcal M'$. Then $Y$ is not invariant for $\mathcal F$.
\end{proposition}
\begin{proof}
It is enough to consider the case when $M=(\mathbb C^n,\mathbf 0)$, with $n\geq 3$, and $E=\emptyset$. Let us take the notations preceding Lemma \ref{lema:dicriticalnotdicritical} and the ones in the proof of that lemma. We are in case {\bf Dic-v}.

Now, let us show that the hypothesis that all the points in $\mathcal M'$ are simple regular points implies that $r=0$ and hence $a_1(\mathbf z)$ is a unit in a generic point of $Y$. This implies that $Y$ is not invariant.
	
Let us write
\begin{eqnarray*}
a_1(\mathbf z)=a_1(x;\mathbf y)=\sum_{s\geq r} A^1_s(x;\mathbf y),
  \\
  p(\mathbf z)=p(x;\mathbf y)= \sum_{s\geq r+1} P_s(x;\mathbf y),
\end{eqnarray*}
where the $A^1_s(x;\mathbf y), P_s(x;\mathbf y)$ are homogeneous polynomials of degree $s$ in the variables $\mathbf y$ and $A^1_r(x;\mathbf y)\ne 0$. Note that
	\begin{eqnarray}
\label{eq:verticaldic}
a_1'(\mathbf z')&=&
	\frac{a_1(\mathbf z)}{{y_2'}^r}=
	\frac{A^1_r(x;\mathbf y)}{{y_2'}^r}+y'_2\left(\cdots\right),\\
\label{eq:verticaldic2}
	a'_2(\mathbf z')&=&\frac{p(\mathbf z)}{{y_2'}^{r+1}}=
	\frac{P_{r+1}(x;\mathbf y)}{{y_2'}^{r+1}}+y'_2\left(\cdots\right), \\
\label{eq:verticaldic3}
	a'_j(\mathbf z')&=&y'_2(\cdots), \; j=3,4,\ldots,n.
\end{eqnarray}
	Assume by contradiction that $r\geq 1$. Up to a change of chart in the blowing-up, we can assume that there is $1\leq t\leq r$ such that
\begin{eqnarray*}
\frac{A^1_r(x;\mathbf y)}{{y'_2}^r}&=& A^1_r(x;1,y'_3,y'_4,\ldots,y'_n)
\end{eqnarray*}
is a polynomial of positive degree $t$ in the variables $y'_3,y'_4,\ldots, y'_n$. We conclude that the set
\begin{equation*}
\label{eq:nosimples}
H=(a'_1(\mathbf z')=y'_2=0)=(A^1_r(x;1,y'_3,y'_4,\ldots,y'_n)=y'_2=0)
\end{equation*}
is non-empty. Let $P$ be a point in $H$. In view of Equations \eqref{eq:verticaldic}, \eqref{eq:verticaldic2} and \eqref{eq:verticaldic3}, the $1$-form $\omega'$ evaluated at $P$ is given by
$$
\omega'(P)=a'_2(P)dy'_2\vert_P.
$$
The hypothesis that $\mathcal F'$ is non-singular at $P$ implies that  $a'_2(P)\ne 0$; but this means that $\mathcal F'$ is tangent at $P$ to the dicritical component $y'_2=0$ and hence $P$ is not a simple regular point for $\mathcal M'$. This is the desired contradiction.

	We conclude that $r=0$ and hence $Y$ is not invariant.
\end{proof}
\begin{proposition}
\label{prop:vertical2}
Let $\mathcal M=(M,E,\mathcal F)$ be a foliated space of dimension $n\geq 3$ and let $\pi:\mathcal M'=(M',E',\mathcal F')\rightarrow \mathcal M=(M,E,\mathcal F)$ be a dicritical admissible blowing-up with center a curve $Y$. Assume that the foliation $\mathcal G=\mathcal F'|_{\pi^{-1}(Y)}$ is not singular and that there is a point $Q\in Y$ such that $\pi^{-1}(Q)$ is invariant for $\mathcal G$. Then, the morphism $\pi$ is a vertical blowing-up.
\end{proposition}
\begin{proof}
It is enough to do the proof in the case when $M=(\mathbb C^n,\mathbf 0)$ and $E=\emptyset$, assuming that $Q=\mathbf 0$. Recall that we are in case {\bf Dic}.
	
Take again the notations developed along this subsection. Moreover, let us write
$$
a_1(x;\mathbf y)=\sum_{s\geq r} A^1_s(x;\mathbf y), \qquad a_j(x;\mathbf y)=\sum_{s\geq r-1} A^j_s(x;\mathbf y),
$$
for $j=2,3,\ldots,n$. Consider the 1-form given by
$$
W=A^1_r(x;\mathbf{y})dx+\sum_{j=2}^n A^j_{r-1}(x;\mathbf{y})dy_j.
$$
Let $\Phi$ be a maximal common factor of the coefficients of $W$ and let us write $W=\Phi\tilde W$. Note that $\Phi$ is homogeneous in the variables $\mathbf{y}$ of a certain degree $0\leq t\leq r$. Let us write
$$
\tilde W=\tilde A^1_{r-t}(x;\mathbf{y})dx+\sum_{j=2}^n \tilde A^j_{r-t-1}(x;\mathbf{y})dy_j.
$$
Let $\mathcal H$ be the foliation defined by $\tilde W=0$. The blowing-up $\pi$ is dicritical for $\mathcal H$ and the induced foliation of the transform $\mathcal H'$ of $\mathcal H$ over the exceptional divisor $\pi^{-1}(Y)$ coincides with $\mathcal G$, that is
$$
\mathcal H'|_{\pi^{-1}(Y)}=\mathcal G.
$$
The reader can verify that the fact that $x=0$ is invariant for $\mathcal G$ implies that $x=0$ is invariant for $\mathcal H$. In particular, the coefficient $\tilde A^1_{r-t}(x;\mathbf{y})$ is not divisible by $x$ and hence it is not identically zero.

Note that if $t=r$, the coefficients of $\tilde W$ for $dy_j$ are identically zero, for $j=2,3,\ldots,n$ and in this case the foliation $\mathcal G$ is given by $dx=0$, that is, the blowing-up $\pi$ is vertical.

Let us show that $t=r$. Assume the contrary, that is $r-t \geq 1$. Since $x$ does not divide $\tilde A^1(x;\mathbf{y})$, there is a monomial belonging to $\tilde A^1(x;\mathbf{y})$ not divisible by $x$. Hence, up to a change of chart in the blowing-up, we can assume that there is $1\leq k\leq r$ such that
\begin{eqnarray*}
\tilde A^1_r(0;1,y'_3,y'_4,\ldots,y'_n)
\end{eqnarray*}
is a polynomial of positive degree $t$ in the variables $y'_3,y'_4,\ldots, y'_n$. Recalling that $x$ divides $A^j_{r-t-1}(x;\mathbf{y})$, for $j=2,3,\ldots,n$, the singular locus of $\mathcal G=\mathcal H'|_{\pi^{-1}(Y)}$ contained in $x=0$ is given by in this chart by the non-empty set
\begin{equation*}
(\tilde A^1_r(0;1,y'_3,y'_4,\ldots,y'_n)=0) \cap(x=y'_2=0).
\end{equation*}
This is the desired contradiction.
\end{proof}
\section{Dimension Three. First Statements}
\label{Dimension Three}
Let $\mathcal M_0=((\mathbb C^3,\mathbf 0),\emptyset,\mathcal F_0)$ be an almost radial foliated space, where the origin is a singular point for $\mathcal F_0$. By definition, there is an adjusted resolution sequence
\begin{equation}
	\label{eq:desingularization sequence}
	\mathcal S:\quad\quad
\mathcal M_0\stackrel{\pi_1}{\leftarrow}
\mathcal M_1
\stackrel{\pi_2}{\leftarrow}
\cdots
\stackrel{\pi_N}{\leftarrow}
\mathcal M_N.
\end{equation}
In this Section \ref{Dimension Three}, we describe some
features of the  sequence $\mathcal S$.

Note that $N\geq 1$, since the origin is a singular point for $\mathcal F_0$.

Let us fix some notations. We put $\mathcal M_j=(M_j,E^j,\mathcal F_j)$ and we denote by $Y_{j-1}\subset M_{j-1}$ the center of $\pi_j$, for any $j=1,2,\ldots,N$. Let us also denote by $E^j_j$ the exceptional divisor $\pi_j^{-1}(Y_{j-1})$ and by $E^j_\ell$ the strict transform in $M_j$ of $E^\ell_\ell$, for $1\leq\ell< j$. In this way, we have that
$$
E^j=\cup_{\ell=1}^j E^j_\ell
$$
is the decomposition of $E^j$ into irreducible components.

Moreover, it is useful to consider the consecutive composition of blowing-ups in $\mathcal S$, so we denote
$$
\pi_j^{j+s}= \pi_{j+1}\circ \pi_{j+2}\circ \cdots \circ \pi_{j+s},
$$
for each $1 \leq s\leq  N-j$. We declare $\pi_j^j$ to be the identity in $\mathcal M_j$; let us note the double notation $\pi_j^{j+1}=\pi_{j+1}$.

\subsection{Equireduction Points}
\label{Equireduction Points}
We give here the notion of {\em $\mathcal S$-equireduction point}. This idea has been introduced in other similar contexts, the reader can see \cite{Can-Mat}.

We define the {\em essential sets $Z_j\subset M_j$} by inverse induction on the index $j=0,1,\ldots,N$ as follows. If $j=N$ we put $Z_N=\emptyset$. For $0\leq j\leq N-1$, we put
$
Z_j=Y_j\cup \pi_{j+1}(Z_{j+1})
$.
Let us note that
$$
Z_j=Y_j^j\cup Y_j^{j+1}\cup Y_j^{j+2}\cup \cdots \cup Y_j^{N-1},
$$
where $Y_j^j=Y_j$ and
$
Y_j^{j+s}=\pi_j^{j+s}(Y_{j+s})$, for $1\leq s<N-j
$.

The essential sets coincide with the singular locus as stated in next Proposition
\ref{prop:essentialsets}:
\begin{proposition}
\label{prop:essentialsets}
  We have that $Z_j=\operatorname{Sing}(\mathcal F_j)$, for any $j=0,1,\ldots,N$.
\end{proposition}
\begin{proof}
Let us first prove that $\operatorname{Sing}(\mathcal F_j)\subset Z_j$. Given a point $Q\in M_{j}\setminus Z_j$, there is an open neighbourhood $V$ of $Q$ such that $\mathcal M_j\vert_{V}$ is isomorphic to the restriction of $\mathcal M_N$ to an open set of $M_N$, via the composition $\pi_j^N$. Since we know that $\operatorname{Sing}(\mathcal F_N)=\emptyset$, we see that $Q$ cannot be a singular point of $\mathcal F_j$.

Conversely, if there is a point $P\in Z_j\setminus  \operatorname{Sing}(\mathcal F_j)$, let $k\geq j$ be the first index such that $P\in Y_{k}$ (up to the local isomorphisms of the previous blowing-ups around $P$). Then $\pi_{k+1}$ is a blowing-up with invariant center not contained in the singular locus. This generates indestructible singularities as stated in Proposition \ref{prop:indestructible}.
\end{proof}
\begin{definition}
\label{def:equireduction}
We say that a point $P\in Z_j$ is an {\em $\mathcal S$-equireduction point} if
for any index $0\leq s < N-j$ and any $P'\in Z_{j+s}$ such that $\pi_j^{j+s}(P')=P$ we have the following properties:
\begin{enumerate}
    \item The germ $(Z_{j+s},P')$ is a non-singular analytic curve of $(M_{j+s},P')$ having full normal crossings with $E^{j+s}$.
    \item
   There is an isomorphism $(Z_{j+s},P')\rightarrow (Z_j,P)$ induced by $\pi_j^{j+s}$.
\end{enumerate}
\end{definition}
\begin{remark}
The complement in $Z_j$ of the set of $\mathcal S$-equireduction points is a finite set.
\end{remark}
\begin{remark} Let $P\in Z_j$ be an $\mathcal S$-equireduction point and consider a point $P'\in Z_{j+s}$ such that $\pi_{j}^{j+s}(P')=P$. Then $P'\in Z_{j+s}$ is also an $\mathcal S$-equireduction point.
\end{remark}
\begin{proposition} \label{prop:equireduction}
Let $P\in Z_j$ be an $\mathcal S$-equireduction point. We have the equality of germs
$$
(Z_j,P)=(Y_j^{j+s},P),
$$ for any integer number $s$ with $0\leq s< N-j$ such that $P\in Y_j^{j+s}$.
\end{proposition}

\begin{proof} Let us consider first the case $s=0$. Hence we assume that $P\in Y_j$ and we are going to show that $(Z_j,P)=(Y_j,P)$.
We know that $(Z_j,P)$ is an non-singular irreducible curve and that $(Y_j,P)\subset (Z_j,P)$. Then, it is enough to show that the germ $(Y_j,P)$ is a germ of curve.  Let us reason by contradiction. If $(Y_j,P)$ is not a germ of curve, then we have that $Y_j=\{P\}$. That is, the morphism $\pi_{j+1}$ is the quadratic blowing-up with center $P$.  Denote by $\Gamma=(Z_j,P)$ the non-singular germ of curve of $Z_j$ at $P$. Let $\Gamma'$ be the strict transform of $\Gamma$ by $\pi_{j+1}$ and denote $P'=\Gamma'\cap \pi_{j+1}^{-1}(P)$. We know that $\Gamma'\subset Z_{j+1}=\operatorname{Sing}(\mathcal F_{j+1})$ and, in particular, we have that $P'\in Z_{j+1}$. The point $P'$ is a $\mathcal S$-equireduction point and then $(Z_{j+1},P')$ is a non-singular irreducible germ of curve. ``A fortiori'', we have the equality of germs
$$
(\Gamma',P')=(Z_{j+1},P').
$$
The contradiction arrives noting that $\Gamma'$, and hence $(Z_{j+1},P')$, is transverse to the exceptional divisor $\pi_{j+1}^{-1}(P)$ and thus the condition of full normal crossings fails.	
	
Let us prove the statement for an integer $s$ such that $1\leq s<N-j$ by inverse induction on the index $j$. If $j=N$, there is nothing to say, since $Z_N=\emptyset$. In the case $j=N-1$ the result is also true, since we have that $Z_{N-1}=Y_{N-1}$. Assume that the result is true for $j'>j$. We have two cases $P\notin Y_j$ or $P\in Y_j$.
	
Assume first that $P\notin Y_j=Y_j^{j}$. Recalling that the blowing-up morphism $\pi_{j+1}$ is centered at $Y_j$, we see that $\pi_{j+1}$ defines  a local isomorphism over the point $P$. Moreover, we have that $P'\in Y_{j+1}^{j+s}$, where $\pi_{j+1}^{-1}(P)=\{P'\}$. By induction hypothesis applied to $j'=j+1$ we have that $(Z_{j+1}, P')=(Y_{j+1}^{j+s},P')$. We conclude since we have that
	$$
	\pi_{j+1}((Z_{j+1}, P'))=(Z_{j}, P), \quad \pi_{j+1}((Y_{j+1}^{j+s},P'))=(Y_{j}^{j+s},P).
	$$
	
Assume that $P\in Y_j$. We know that $(Z_j,P)$ is an non-singular irreducible curve. Since $P\in Y_j^{j+s}$, we have that $(Y_j^{j+s},P)\subset (Z_j,P)$. Then, it is enough to show that the germ $(Y_j^{j+s},P)$ is a germ of curve. Take a point $P'\in Y_{j+s}$ that projects over $P$, that is $\pi_j^{j+s}(P')=P$. Since
	$$
	\pi_j^{j+s}((Y_{j+s},P')) \subset (Y_j^{j+s},P),
	$$
	it is enough to show that $\pi_j^{j+s}((Y_{j+s},P'))$ is a curve. By induction hypothesis, we have that $(Y_{j+s},P')=(Z_{j+s},P')$. Since $P$ is an $\mathcal S$-equireduction point, we know that $(Z_{j+s},P')$ is a germ of curve isomorphic to $(Z_{j},P)$  via $\pi_j^{j+s}$. We conclude that $\pi_j^{j+s}((Y_{j+s},P'))$ coincides with $(Z_j,P)$ and then it is a germ of curve, as desired.
\end{proof}
\begin{remark}
As a consequence of this Proposition  \ref{prop:equireduction}, we have that $(Y_{j}^{j+s},P)$ is a germ of non-singular irreducible curve, when $P\in Y_{j}^{j+s}$ and it is an $\mathcal S$-equireduction point.
\end{remark}

\subsection{Index-Persistent Equireduction Points} 
We introduce here a kind of equireduction points that are persistent under blowing-ups.

\begin{lema}
	\label{lema:indexpoints}
Consider an $\mathcal S$-equireduction point $P\in Z_j=\operatorname{Sing}(\mathcal F_j)$ and assume that $P$ belongs to a non-dicritical component $D$ of $E^j$. Let $(\Delta,P)\subset (M,P)$ be germ of non-singular surface transverse to $Z_j$. Then, we have that $\Delta$ is transverse to $\mathcal F_j$ at $P$ and thus we have a well-defined restriction $\mathcal G=\mathcal F_j \vert_\Delta$. Moreover, the curve $\Delta\cap D$ is invariant for $\mathcal G$.
\end{lema}
\begin{proof}
If $\Delta$ is invariant for $\mathcal F_j$, then the curve $\Delta \cap D$ is contained in the singular locus $\operatorname{Sing}(\mathcal F_j)=Z_j$. This is a contradiction with the transversality condition. To see that $\Delta \cap D$ is invariant for $\mathcal G$ it is enough to look at point $Q\in \Delta \cap D\setminus Z_j$.
\end{proof}
 An $\mathcal S$-equireduction point $P\in Z_j$ 
 is an {\em index-persistent equireduction point for $\mathcal M_j$} if $P$ belongs to a non-dicritical component $D$ of $E^j$ and there is a germ of  non-singular  surface $(\Delta,P)\subset (M,P)$ such that:
\begin{enumerate}
\item $\Delta$ is transverse to $Z_j$. Hence, by Lemma \ref{lema:indexpoints}, the restricted foliation  $\mathcal G=\mathcal F_j \vert_\Delta$ exists and the curve $\Delta\cap D$ is invariant for $\mathcal G$.
\item  The Camacho-Sad
index of $\mathcal G$ with respect to $\Delta\cap D$ at the point $P$ has strictly negative real part.
\end{enumerate}
\begin{remark}
We ask $P$ to be and index-persistent point for $\mathcal G$, but additionally we fix the invariant curve $\Delta \cap D$ to get a Camacho-Sad index with negative real part.
\end{remark}

\begin{proposition}
\label{prop:stability of index persistent}
Assume that there is an index-persistent equireduction point for $\mathcal M_j$. Then, there is also an  index-persistent equireduction point for $\mathcal M_{j+1}$.
\end{proposition}
\begin{proof}
Without lost of generality we can assume that $P\in Y_j$. Let $P'$ be the point $P'=\pi_{j+1}^{-1}(P) \cap D'$, where $D'$ is the strict transform of $D$ by $\pi_{j+1}$. We know that Camacho-Sad index of $\mathcal G'$ at $P'$ with respect to $\Delta'\cap D'$ has strictly negative real part, see Remark \ref{remark:index-persistent}. Now, it is enough to see that $P'$ is a singular point $P'\in Z_{j+1}$, hence, an equireduction point and thus, an index-persistent equireduction point. If $P'\notin Z_{j+1}$, we would have that $\mathcal G'$ is regular at $P'$ and then, the Camacho-Sad index should be zero.
\end{proof}
 \begin{corollary}
 	\label{cor:noindexpersistent}
 	There are no index-persistent equireduction points for $\mathcal M_j$, for any $j=0,1,\ldots,N$.
\end{corollary}
\begin{proof}
	Since $\operatorname{Sing}(\mathcal F_N)=\emptyset$, there are no index-persistent equireduction points for $\mathcal M_N$. Now, we apply the proposition to see that there are no index-persistent equireduction points for $\mathcal M_j$, for $j=0,1,\ldots,N$.
\end{proof}

\subsection{Complete Dicriticalness} This subsection is devoted to prove the following statement:
\begin{proposition}
\label{prop:totaldicriticalness}
All the blowing-ups in $\mathcal S$ are dicritical blowing-ups.
\end{proposition}
We prove Proposition \ref{prop:totaldicriticalness} as a consequence of several propositions.
\begin{proposition}
\label{prop:totaldicriticalnessquadratic}
If the first blowing-up $\pi_1:\mathcal M_1\rightarrow\mathcal M_0$ is quadratic, then it is dicritical.
\end{proposition}
\begin{proof}
Assume, by contradiction, that $\pi_1$ is non-dicritical and hence the exceptional divisor $E^1=\pi_1^{-1}(\mathbf 0)$ is invariant for $\mathcal F_1$. Let us recall that $E^1$ is isomorphic to the projective plane $\mathbb P^2_{\mathbb C}$.

Let us consider the intersection $Z_1\cap E^1$ of the exceptional divisor $E^1$ with the set $Z_1=\operatorname{Sing}(\mathcal F_1)\subset M_1$. Recall that $Z_1\cap E^1$ is a closed analytic subset of $E^1$ of dimension at most one, then it is a (maybe empty) finite union of points and curves contained in $E^1$. Moreover, the set $T\subset Z_1\cap E^1$ of non-$\mathcal S$-equireduction points in $Z_1\cap E^1$ is a finite set.

Select a non-singular surface $(\Delta,\mathbf 0)\subset (\mathbb C^3,\mathbf 0)$ whose strict transform $\Delta'$ by $\pi_1$ satisfies the following properties:
\begin{enumerate}[a)]
\item $T\cap L'=\emptyset$, where $L'=\Delta'\cap E^1$.
\item The projective line $L'$ cuts $Z_1\cap E^1$ in a transverse way.
\end{enumerate}
Indeed, the set of projective lines satisfying the above two properties is a non-empty Zariski open set.

 Let us show that $\Delta'$ is not invariant for $\mathcal F_1$ and hence we also have that $\Delta$  is not invariant for $\mathcal F_0$. Note that $L'\cap Z_1$ is a finite set.  Consider a point $Q\in L'\setminus Z_1$. Recall that we are assuming that $E^1$ is invariant for $\mathcal F_1$. Since $Q$ is a regular point for $\mathcal F_1$ and $\Delta'$ is transverse to $E^1$,  we deduce that $\Delta'$ is not invariant for $\mathcal F_1$.

 Let us consider the restrictions
$$
\mathcal G=\mathcal F_0\vert_{\Delta}; \quad \mathcal G'=\mathcal F_1\vert_{\Delta'}.
$$
Looking at a point $Q\in L'\setminus Z_1$ as before, we deduce that $L'$ is invariant for $\mathcal G'$. In other words, the blowing-up
$$
(\Delta,\emptyset, \mathcal G)\leftarrow (\Delta', L',\mathcal G')
$$
is a non-dicritical blowing-up. Since the self-intersection  of $L'$ in $\Delta'$ is equal to $-1$, there is at least one point $P\in L'$ such that the real part of Camacho-Sad index of $\mathcal G'$ with respect to $L'$ at $P$ is strictly negative. This is an index-persistent equireduction point, contradiction by Corollary \ref{cor:noindexpersistent}
\end{proof}

\begin{remark} \label{rk:finalblowing-up}
In the quadratic case, we always have $N \geq 2$. Indeed, if $N$ would be equal to $1$, we would have and induced foliation over the projective plane without singularities and this is not possible.
\end{remark}

Now, we go to the case of a monoidal first blowing-up.
\begin{proposition}
\label{prop:firstmonoidal}
 Assume that the center $Y_0$ of the first blowing-up $\pi_1$ in the sequence $\mathcal S$ is a curve.  Then $\pi_1$ is dicritical.
\end{proposition}
\begin{proof} The proof follows with arguments very similar to the ones in the proof of Proposition \ref{prop:totaldicriticalnessquadratic}. Let us give the main ideas.
	
By taking an appropriate representative of the germ $Y_0$, we can consider an equireduction point $Q_0\in Y_0$ close to the origin. In order to prove that $E^1=\pi^{-1}(Y_0)$ is dicritical for $\mathcal M_1$, it is enough to localize the sequence $\mathcal S$ at $Q_0$. Then, without loss of generality, we can assume that the origin $\mathbf 0$ is an $\mathcal S$-equireduction point.

Let us start our argument by contradiction, assuming that $D=E^1_1$ is non-dicritical for $\mathcal M_1$.
Take a regular two-plane $\Delta\subset (\mathbb C^2,\mathbf 0)$ transverse to $Y_0$. Let $\Delta'$ be the strict transform of $\Delta$ by $\pi_1$. The curve
$$
L'=\pi_1^{-1}(\mathbf 0)=\Delta'\cap E^1
$$
is invariant for $\mathcal G'=\mathcal F_1\vert_{\Delta'}$. Moreover, there is a singular point $P'$ for $\mathcal G'$ such that the Camacho-Sad index of $\mathcal G'$ at $P'$ with respect to $L'$ has strictly negative real part. The point $P\in Z_1$, by arguments as in the proof of Proposition \ref{prop:stability of index persistent}, hence it is an index-persistent equireduction point and we contradict Corollary \ref{cor:noindexpersistent}.
\end{proof}

Now we end the proof of Proposition \ref{prop:totaldicriticalness} by induction on $N$. If $N=1$, the result is true. If $N\geq 2$,  we can re-start at the foliated space obtained from $\mathcal M_1$ by skipping the dicritical divisor $E^1$, in view of
 the results in Subsection \ref{Radial Foliated Spaces}.

\begin{remark}
Note that the last blowing-up $\pi_N$ cannot be a quadratic one, by the same arguments as in Remark \ref{rk:finalblowing-up}.
\end{remark}
\subsection{Corners of the Divisor}
\label{Corners of the divisor}

Let $(M,E)$ be an ambient space of dimension $n$. A point $P\in M$ is called to be a {\em corner for $(M,E)$} when the number of irreducible components of $E$ through $P$ is equal to $n$.
\begin{remark}
	\label{stability of corners}
	Let $(M',E')\rightarrow (M,E)$ be a blowing-up of ambient spaces. If there is a corner point for $(M,E)$, then there is also a corner point for $(M',E')$, Equivalently, if $(M',E')$ is without corners, then $(M,E)$ is also without corners.
\end{remark}

If $\mathcal M=(M,E,\mathcal F)$ is a foliated space, a {\em dicritical corner for $\mathcal M$} is, by definition, a corner for the ambient space $(M,E_{\operatorname{dic}})$, where $E_{\operatorname{dic}}$ is the union of the dicritical components of $E$. By Remark \ref{rk:nodicritical corners}, we know that if all the points in $\mathcal M$ are regular simple points, there are no dicritical corners.

\begin{proposition}
 \label{prop:nocorners}
 Consider the sequence $\mathcal S$ in Equation \eqref{eq:desingularization sequence}. Then each ambient space $(M_i,E^i)$ is without corners.
\end{proposition}
\begin{proof} It is enough to show that $(M_N,E^N)$ is without corners. Recall that $\mathcal M_N$ has only simple regular points and that all the irreducible components of $E^N$ are dicritical ones. In this situation, no corners are possible.
\end{proof}
\begin{remark}
The above Proposition \ref{prop:nocorners} gives indications on the blowing-ups underlying the sequence $\mathcal S$. Namely,
\begin{enumerate}
	\item If the center $Y_{i-1}$ is a point, then it cannot be in the intersection of two components of the divisor $E^{i-1}$.
	\item  If the center  $Y_{i-1}$ is a curve with $Y_{i-1}\subset E^{i-1}$, then there is no component of $E^{i-1}$ transverse to $Y_{i-1}$. That is, the curve $Y_{i-1}\subset E^{i-1}$ has full normal crossings with $E^{i-1}$.
\end{enumerate}
\end{remark}

\subsection{Restrictions to the Components of the Divisor}
\label{Restrictions to the Components of the Divisor}
Consider an intermediate step  $\mathcal M_j=(M_j,E^j,\mathcal F_j)$ in the sequence $\mathcal S$ and let $D$ be an irreducible component of the divisor $E^j$. We know that $D$ is a dicritical component of $\mathcal M_j$.  Hence, we get a two-dimensional foliated space
$$
\mathcal M_j\vert_D=(D,E^j\vert_D,\mathcal F_j\vert_D),
$$
obtained by restriction of $\mathcal M_j$ over $(D,E^j\vert_D)$. The normal crossings divisor $E^j\vert_D$ is given by the intersections with $D$ of the components of $E^j$ not equal to $D$.
\begin{lema}
\label{lema:restriccion}
The foliated space $\mathcal M_j\vert_D$ is a two-dimensional radial foliated space and all the irreducible components of $E^j\vert_D$ are dicritical components for $\mathcal M_j\vert_D$.
\end{lema}
\begin{proof} The divisor $D$ induces a restricted sequence $\mathcal S\vert_D$ that provides a resolution sequence for $\mathcal M_j\vert_D$. The dicriticalness of the components of $E^j\vert_D$ can be tested in the last step of $\mathcal S$.
\end{proof}

Note that $\operatorname{Sing}(\mathcal F_j|_D)$ is a finite subset of $D$ and that $\mathcal F_j|_D$ is a  cart-wheel foliation at these points.

\begin{proposition} \label{prop:restriccionextendidadacomponentes}
	Consider an index $0\leq j\leq N$ and let $D$ be an irreducible component of the divisor $E^j$. Let us denote by $H$ the (finite) union of the irreducible curves contained in $\operatorname{Sing}(\mathcal F_j)\cap D=Z_j\cap D$. The following statements hold:
	\begin{enumerate}
		\item The divisor $\tilde H=E^{j}\vert_D\cup H$ is a normal crossings divisor of $D$, the irreducible components of  $\tilde H$ are dicritical for the restriction
		$
		\mathcal G=\mathcal F_j|_D
		$
		and  $\mathcal N_{j,D}=(D,\tilde H, \mathcal G)$ is a radial foliated space of dimension two.
		\item If $j<N$ and $Y_{j}\cap H\ne\emptyset$, then $Y_{j}$ is a curve contained in $H$.
	\end{enumerate}
\end{proposition}

\begin{proof}Induction on $N-j$.  If $j=N$ we are done, since $\mathcal M_N$ is a desingularized foliated space without singularities, then $H=\emptyset$ and we are done. Assume that $j<N$. Let us first prove Statement (2). We do an argument by contradiction, assuming that $Y_{j}\cap H\ne\emptyset$ is a single point $Q$ (actually, we obtain the same contradiction when $Y_{j}\cap H$ is a non-empty finite set of points). The morphism $\pi_{j+1}$ induces a blowing-up
	$$
	D'\rightarrow D,
	$$
	centered at the point $Q$, where $D'$ is the strict transform of $D\subset M_{j}$ by $\pi_{j+1}$. Invoking induction hypothesis, we see that this blowing-up will create a dicritical corner in
	$\mathcal N_{j+1,D'}$. This is not possible for a two dimensional radial foliated space.
	
Let us see that Statement (2) implies Statement (1). We have the following possibilities:
\begin{enumerate}
\item $Y_j \cap D=\emptyset$. In this case $\mathcal N_{j,D}$ is identical to $\mathcal N_{j+1,D'}$.
\item $Y_j \cap D\ne \emptyset$ and $Y_j \cap H\ne \emptyset$. By Statement (2) we obtain that $Y_j$ is a curve contained in $H$. Then $\mathcal N_{j,D}$ is also identical to $\mathcal N_{j+1,D'}$.
\item $Y_j \cap D\ne \emptyset$ and $Y_j \cap H= \emptyset$. In this case  $Y_j \cap D$ is a finite union of points. Indeed, if $Y_j \cap D$ is a curve, it is contained in $H$. Now, the new created exceptional divisors are disjoint from the strict transform of $H$. We deduce immediately the desired properties for $\mathcal N_{j,D}$ from the ones of $\mathcal N_{j+1,D'}$, that are guaranteed by induction hypothesis.
	\end{enumerate}
	The proof is ended.
\end{proof}

\begin{corollary}
\label{cor:comonenteslugarsingulareneldivisor}
Consider an index $1\leq j\leq N-1$. The set
$
E^j\cap\operatorname{Sing}(\mathcal F_j)
$
is a finite union of points and non-singular irreducible curves mutually disjoint.
\end{corollary}
\begin{proof}
Let $Y$ be a curve contained in $E^j\cap\operatorname{Sing}(\mathcal F_j)$. There is an irreducible component $D$ of $E^j$ such that $Y\subset D$.
 In view of Proposition \ref{prop:restriccionextendidadacomponentes}, we know that $Y$ is an irreducible component of the normal crossings divisor $E^{j}\vert_D\cup H$, where $H$ is the union of the curves in the singular locus contained in $E^j$. Hence, the curve $Y$ is non-singular.

Assume that there is another curve $\Gamma\ne Y$ in $E^j\cap\operatorname{Sing}(\mathcal F_j)$ such that $\Gamma\cap Y\ne \emptyset$ and let us find a contradiction.

If $\Gamma\subset D$, both $\Gamma$ and $Y$ are mutually intersecting dicritical components of the normal crossings divisor $E^{j}\vert_D\cup H$. Then we have a dicritical corner of in the radial foliated space $\mathcal N_{j,D}$ and this is not possible.

Assume that $\Gamma\not\subset D$. Given a point $P$ in $\Gamma \cap Y$, we know that it is a regular point for $\mathcal N_{j,D}$ since it belongs to the dicritical component $Y$ of the radial foliated space $\mathcal N_{j,D}$. 

The modification of $\Gamma$ by the subsequent blowing-ups induces a quadratic blowing-up for
$\mathcal N_{j,D}$ centered at the point $P$. This would generate an indestructible singularity.
\end{proof}
\subsection{Compact Curves of the Singular Locus}
 \label{Compact Curves of the Singular Locus}
 In this subsection we show that there are no compact curves in the singular locus.

\begin{lema}
\label{lema:nohirzebruchtubes}
Consider an index $0\leq j\leq N$. There are no Hirzebruch tubes $(M_j,E^j;Y)$ such that $Y$ is a curve in the singular locus of $\mathcal F_j$.
\end{lema}
\begin{proof}
We do the proof by induction on $N-j$. If $j=N$, we are done, since the singular locus of $\mathcal F_N$ is empty. Assume that $j<N$, jointly with the corresponding induction hypothesis and let us find a contradiction with the existence of a Hirzebruch tube $(M_j,E^j;Y)$ such that $Y$ is a curve in the singular locus of $\mathcal F_j$.

If $Y_j\cap Y=\emptyset$, we are done, since the blowing-up $\pi_{j+1}$ induces the identity outside $Y$ and hence we find a Hirzebruch tube
$$
(M_{j+1},E^{j+1};Y)
$$
such that $Y$ is a curve in the singular locus of $\mathcal F_{j+1}$; contradicting the induction hypothesis.
	
Assume now that $Y_j\cap Y\ne\emptyset$. Let us note that $Y$ is contained in $E^j\cap \operatorname{Sing}(\mathcal F_j)$. Moreover, in view of Corollary
\ref{cor:comonenteslugarsingulareneldivisor}, the curve $Y$ does not intersect any other irreducible component of $E^j\cap\operatorname{Sing}(\mathcal F_j)$.

By application of Statement (2) in Proposition \ref{prop:restriccionextendidadacomponentes}, we necessarily have that $Y_j\subset E^j\cap \operatorname{Sing}(\mathcal F_j)$ and then $Y=Y_j$. Let us perform the blowing-up
$
\pi_{j+1}:\mathcal M_{j+1}\rightarrow \mathcal M_j
$
with center $Y_j=Y$. 

We will apply now Proposition \ref{prop:verticalblowingups}, to get a contradiction. In order to do it, we will show that $\pi_{j+1}$ is a vertical blowing-up and that $Z=\emptyset$, where
$$
Z=\pi_{j+1}^{-1}(Y)\cap \operatorname{Sing}(\mathcal F_{j+1}).
$$
By Proposition \ref{prop:verticalblowingups} we obtain the contradiction that $Y$ is not invariant.
	
Let us first show that $\pi_{j+1}$ is a vertical blowing-up. We know that the only radial foliation over the Hirzebruch surface $\pi_{j+1}^{-1}(Y)$ is given by the fibers of the points of $Y$ by $\pi_{j+1}$, see Subsection \ref{Radial Foliations over Hirzebruch Surfaces}. Then, this foliation must coincide with the restriction $\mathcal G$ of $\mathcal F_{j+1}$ to the exceptional divisor $\pi_{j+1}^{-1}(Y)$, in view of Lemma \ref{lema:restriccion}. Hence the blowing-up $\pi_{j+1}$ is a vertical blowing-up.
	
Let us show now that $Z=\emptyset$. Assume the contrary.  If $Z$ contains an isolated point, this point induces a quadratic blowing-up in the surface $\pi_{j+1}^{-1}(Y)$ and then, it should create an indestructible singularity in the foliation $\mathcal G$. Hence $Z$ is a finite union of curves. Take an irreducible component $Y'$ of $Z$. By applying Proposition \ref{prop:restriccionextendidadacomponentes}, we have the following properties concerning $Y'$:
\begin{enumerate}
	\item $Y'$ cannot coincide with a fiber of the blowing-up, since it must be dicritical.
	\item $Y'$ is not tangent to any fiber.
	\item $Y'$ has full normal crossing with $E^{j+1}$.
\end{enumerate}
We get that $Y'$ is a $1$-infinitely near curve of $Y$, then by Statement c) of Proposition \ref{prop:alphabetaHirzebruch1} we obtain a Hirzebruch tube
$$
(M_{j+1},E^{j+1};Y').
$$
Since $Y'$ is a curve of the singular locus of $\mathcal F_{j+1}$, we reach a contradiction with the induction hypothesis.
\end{proof}
\begin{proposition}
	\label{prop:nocompactcurves}
	The singular locus  of $\mathcal F_j$ does not contain compact curves, for any $j=0,1,\ldots,N$.
\end{proposition}
\begin{proof}
Assume by contradiction that there is an index $0\leq j\leq N-1$ such that there is a  compact curve $Y\subset M_{j}$ contained in the singular locus $\operatorname{Sing}(\mathcal F_{j})$ and that the singular locus $\operatorname{Sing}(\mathcal F_{k})$ does not contain compact curves for $0\leq k<j$.  If we show that $(M_{j},E^{j};Y)$ is a Hirzebruch tube, we obtain the desired  contradiction by application of Lemma \ref{lema:nohirzebruchtubes}.
	
Note that we necessarily have that $j\geq 1$ since $M_0=(\mathbb C^3,\mathbf 0)$ does not contain any compact curve. Moreover, the curve $Y$ is contained in
$$
E^{j}_{j}=D=\pi_{j}^{-1}(Y_{j-1}).
$$
Let us recall that $Y_{j-1}$ is contained in the singular locus  of $\mathcal F_{j-1}$, hence  $Y_{j-1}$ is not a compact curve. Then, there is a point $Q\in \operatorname{Sing}(\mathcal F_{j-1})$ such that one of the following situations occurs:
\begin{enumerate}
	\item The center $Y_{j-1}$ is the germ of a curve at the point $Q$.
	\item $Y_{j-1}=\{Q\}$.
\end{enumerate}
In both cases, we have that $Y\subset \pi_{j}^{-1}(Q)$.
	
Assume first that we are in situation (1), that is $\pi_{j}$ is a monoidal blowing-up with center at the germ of curve $Y_{j-1}$ at the point $Q$. In this case, we have that $\pi_{j}^{-1}(Q)$ is the only compact curve contained in $D$. Then, we have that
$$
Y=\pi_{j}^{-1}(Q).
$$
Note that the germ of curve $Y_{j-1}$ is not contained in any irreducible component of $E^{j-1}$. Otherwise, the divisor of $D$ given by
$$
E^j|_{D} \cup Y
$$
would have a corner, contradicting Proposition \ref{prop:restriccionextendidadacomponentes}. Then, by taking local coordinates $(x,y,z)$ at $Q$, we can assume that
$$
E^{j-1}\subset (x=0), \quad Y_{j-1}=(y=z=0)
$$
and considering the standard charts of the blowing-up, we obtain the desired Hirzebruch tube.
	
Assume now that we are in situation (2). We have that  $D=\pi_{j}^{-1}(Q)$ is isomorphic to the projective plane $\mathbb P_{\mathbb C}^2$. Doing a computation in coordinates of the blowing-up, we see that if $Y$ is a projective line, then
$$
(M_{j},H;Y)
$$
is a Hirzebruch tube, where $H$ is the union of the irreducible components of $E^{j}$ containing $Y$ or not intersecting $Y$. Thus,  we have to verify that $Y$ is a projective line of the projective plane $D$ and that $E^j=H$.

Again, in view of Proposition \ref{prop:restriccionextendidadacomponentes}, we know that
$$
(D, E^{j}|_{D}\cup Y,\mathcal F_{j}|_{D}), 
$$
is a plane radial foliated space and that all the irreducible components of $E^{j}\vert_{D}\cup Y$ are dicritical ones. Then by Proposition \ref{prop:radialinprojectiveplanes}, we know that $Y$ is a projective line and that $E^{j}\vert_{D}$ is either the empty set, or it is equal to $Y$. In this way, we obtain that $H=E^{j}$.
\end{proof}

By the previous Proposition \ref{prop:nocompactcurves}, we have that
the centers $Y_j$ in the sequence $\mathcal S$ are either points or germs of curves contained in the singular locus of $\mathcal F_j$, for $j=0,1,\ldots,N-1$.

\section{Radial Foliated Spaces in Dimension Three}
\label{Radial Foliations in dimension Three}

In this Section \ref{Radial Foliations in dimension Three} we give a proof of the main result Theorem \ref{main} in this paper. Below we recall its statement:
\begin{quote}
\em``
Consider a foliated space $\mathcal M_0=((\mathbb C^3,\mathbf 0),E^0,\mathcal F_0)$, with non-empty singular locus. Then $\mathcal M_0$ is a radial foliated space if and only if there are coordinates $x,y,z$ such that $\mathcal F_0$ is the ``open book'' foliation given by
$\omega=ydz-zdy$
and $E^0\subset (xyz=0)$.''
\end{quote}

In order to start the proof of Theorem \ref{main}, we take an $E^0$-controlled adjusted resolution sequence $\mathcal S$ for $\mathcal M_0$ as follows:
\begin{equation}
	\label{eq:desingularization sequence2}
	\mathcal S:\quad\quad
	((\mathbb C^3,\mathbf 0),E^0,\mathcal F_0)=\mathcal M_0\stackrel{\pi_1}{\leftarrow}
	\mathcal M_1
	\stackrel{\pi_2}{\leftarrow}
	\cdots
	\stackrel{\pi_N}{\leftarrow}
	\mathcal M_N.
\end{equation}
Note that $N\geq 1$. Recall that the fact that $\mathcal S$ is $E^0$-controlled means that if the first blowing-up $\pi_1$ is centered in a curve $Y_0$, then there is an irreducible component $D$ of $E^0$ transverse to $Y_0$.

We are going to deal with the following cases:
\begin{enumerate}[A)]
	\item The first blowing-up $\pi_1$ is monoidal, centered at a curve $Y_0$. The starting divisor $E^0$ has a single irreducible component, it is dicritical for $\mathcal M_0$ and transverse to $Y_0$.
	\item The first blowing-up $\pi_1$ is quadratic, centered at the origin  $\mathbf 0\in \mathbb C^3$ and  $E^0=\emptyset$.
	\item The general case.
\end{enumerate}

\subsection{Mattei-Moussu Sections}
\label{Mattei-Moussu Sections}
For more details on the results in this Subsection \ref{Mattei-Moussu Sections}, the reader can look at \cite{Mat-M} and \cite{Mat}.

 Let $\omega$ be an integrable $1$-form over $({\mathbb C}^3,\mathbf 0)$ that we write as
$$
\omega=a(x,y,z)dx+b(x,y,z)dy+c(x,y,z)dz.
$$
Let us assume that $\operatorname{Sing}(\omega)$ is a non-empty analytic subset of $({\mathbb C^3,\mathbf 0})$ of codimension at least two. That is, the $1$-form $\omega$ is a local generator for a foliation $\mathcal F$ such that $\operatorname{Sing}(\mathcal F)\ne\emptyset$. We say that the plane $x=0$ gives a {\em Mattei-Moussu section for $\omega$, (alternatively: for the foliation $\mathcal F$}) if and only if the two variables $1$-form
$$
\eta=\omega\vert_{x=0}=b(0,y,z)dy+c(0,y,z)dz
$$
has isolated singularity in $(\mathbb C^2,\mathbf 0)$. Hence $\eta$ is a local generator for the foliation $\mathcal F\vert_{x=0}$. We will also call $\mathcal F\vert_{x=0}$ the {\em Mattei-Moussu section} of $\mathcal F$ given by $x=0$.

We are interested in considering Mattei-Moussu sections that give cart-wheel foliations.

\begin{lema}
\label{lema:mattei-moussu section}
Take an integrable germ of $1$-form $\omega$ over $({\mathbb C^3},\mathbf 0)$ and let us assume that
$\omega\vert_{x=0}$ can be written as
$$
\omega\vert_{x=0}=ydz-zdy.
$$
Then, up to a coordinate change, we can express $\omega$ as
$$
\omega=u(x,y,z)(ydz-zdy),
$$
 where $u(x,y,z)$ is a unit. Hence, the foliation $\omega=0$ is an open book foliation in $({\mathbb C^3},\mathbf 0)$.
\end{lema}
\begin{proof}
The $1$-form $\omega$ is written as
$$
\omega= ydz-zdy+\alpha(x,y,z)dx+x(\beta(x,y,z)dy+\gamma(x,y,z)dz).
$$
The differential $d\omega$ of $\omega$ is given by
$$
d\omega=\phi_1(x,y,z)dy\wedge dz+\phi_2(x,y,z)dz\wedge dx+\phi_3(x,y,z)dx\wedge dy,
$$
where
$
\phi_1= 2+x\left({\partial \gamma}/{\partial y}-{\partial \beta}/{\partial z}\right)
$, hence $\phi_1$ is a unit. Consider the non-singular germ of vector field
$$
\xi=\phi_1\partial /\partial x+ \phi_2\partial /\partial y+\phi_3\partial /\partial z.
$$
By the integrability property of $\omega$, we have that $\omega(\xi)=0$. On the other hand, the classical rectification of $\xi$ allows us to assume that
$$
\xi=\partial/\partial x,
$$
without loosing the property that $\omega\vert_{x=0}=ydz-zdy$. In this new coordinates, we have that $\alpha=0$, that is, we have
$$
\omega= ydz-zdy+x(\beta(x,y,z)dy+\gamma(x,y,z)dz).
$$
By applying once more the integrability condition, we obtain that
$$
\omega=(1+x\delta(x,y,z))(ydz-zdy),
$$
as desired.
\end{proof}
\begin{corollary}
 \label{cor:libroabierto}
 Let $\mathcal F$ be  a germ of foliation on $({\mathbb C}^3,\mathbf 0)$. Then $\mathcal F$ is an open book foliation if and only if there is a Mattei-Moussu section $\mathcal G=\mathcal F\vert_\Delta$ such that $\mathcal G$ is a cart-wheel foliation. In this case, any plane section transverse to the singular locus is a Mattei-Moussu section.
\end{corollary}

\subsection{First Monoidal Blowing-up}
\label{First Monoidal Blowing-up}  We consider here the case A) above. We have the following proposition:
 \begin{proposition}
 \label{prop:thecaseA}
   Assume that the first blowing-up $\pi_1$ in the resolution sequence $\mathcal S$ is centered at a germ of curve $Y_0$ and that $E^0$ has a single component, which is dicritical for $\mathcal M_0$ and transverse to $Y_0$. Then $E^0$ defines a Mattei-Moussu section $\mathcal F_0|_{E^0}$ of $\mathcal F_0$ that is a cart-wheel foliation.
 \end{proposition}
 \begin{proof}
 Let us first show that $N=1$. Denote $E^1=E^1_0\cup E^1_1$ the divisor in the step $1$, where $E^1_1=\pi_1^{-1}(Y_0)$ and $E^1_0$ is the strict transform of $E^0$. We know that both $E^1_0$ and $E^1_1$ are dicritical components. Moreover, the compact curve $E^1_0\cap E^1_1$ is not in the singular locus of $\mathcal F_1$, in view of the results in Subsection \ref{Compact Curves of the Singular Locus}. Thus, the next center $Y_1$ of $\pi_2$ is either a germ of curve over a point $Q$ or just a point $Q$, where $Q\in E^1_0\cap E^1_1$. In both cases we create a corner of dicritical components and this is not possible. This shows that $N=1$.

 Consider the plane $E^0\subset ({\mathbb C^3},\mathbf 0)$. We know that $E^0$ is dicritical for $\mathcal F_0$ and hence the restriction $\mathcal G=\mathcal F_0\vert_{E^0}$ exists. Moreover, there are no curves of tangencies between $E^0$ and $\mathcal F_0$; otherwise, these curves would be visible after the blowing-up $\pi_1$. Since all the points of $M_1$ are regular simple points for $\mathcal M_1$, we have that the tangency curves between $\mathcal F_1$ and $E^1_0$ do not exist. Then $\mathcal G$ is a Mattei-Moussu section of $\mathcal F_0$. It is desingularized without singularities after a single blowing-up and hence it is a cart-wheel foliation.
 \end{proof}
 The above Proposition \ref{prop:thecaseA} completes the proof of Theorem \ref{main} for case A), in view of Corollary \ref{cor:libroabierto}.

\subsection{First Quadratic Blowing-up} We consider in this subsection the case B) above. That is, we assume that the first center $Y_0$ in the resolution sequence $\mathcal S$ is the origin $Y_0=\{\mathbf 0\}$ and $E^0=\emptyset$. Thus, the sequence $\mathcal S$ fulfills the conditions of the adjusted resolution sequence considered in Section \ref{Dimension Three}.

Let us recall that the exceptional divisor $E^1=E^1_1=\pi_1^{-1}(\mathbf 0)
$
of $\pi_1$ is isomorphic to the projective plane $\mathbb P^2_{\mathbb C}$.
Note also that we already know that $E^1_1$ is a dicritical component for  $\mathcal M_1$.
\begin{proposition}
\label{prop:thecaseB}
Assume that the first blowing up $\pi_1$ in the resolution sequence $\mathcal S$ is a  quadratic blowing-up and that  $E^0=\emptyset$.  Then, there is a Mattei-Moussu section $\mathcal G=\mathcal F_0\vert_{\Delta}$ of $\mathcal F_0$ such that $\mathcal G$ is a cart-wheel foliation.
\end{proposition}
\begin{proof} We already know the following properties:
\begin{enumerate}
\item The restriction $\mathcal G_1=\mathcal F_1\vert_{E^1}$ gives a radial foliated space over $E^1$. Hence, it is the degree zero foliation on the projective plane $E^1$ given by the projective lines through a point $P_1$.
\item The intersection of the singular locus $\operatorname{Sing}(\mathcal F_1)$ with the exceptional divisor $E^1$ is the singleton $\{P_1\}$.  Indeed, it does not contain curves, since they will be necessarily compact curves. Hence, this intersection is a finite set of points. Moreover, modifying points that are not singular for the foliation $\mathcal G_1$ will produce indestructible singularities.
 \end{enumerate}
Let us select a plane $\Delta\subset (\mathbb C^3,\mathbf 0)$ such that the projective line $L=E^1\cap \Delta'$ does not contain $P_1$, where $\Delta'$ is the strict transform of $\Delta$ by $\pi_1$.

In this situation $\Delta'$ cuts transversely $\mathcal F_1$ and the restriction $\mathcal G'=\mathcal F_1\vert_{\Delta'}$ is a regular foliation transverse to $L$. This implies both that $\mathcal G$ is a Mattei-Moussu section of $\mathcal F_0$ and that $\mathcal G$ is a cart-wheel foliation.
\end{proof}
 The above Proposition \ref{prop:thecaseB} completes the proof of Theorem \ref{main} for case B), in view of Corollary \ref{cor:libroabierto}.
\subsection{Dicriticalness of a Transverse Component}
\label{Dicriticalness of a Transverse Component}
 Before going to the general case C), let us consider the case when $\pi_1$ is monoidal. Let us recall that the blowing-up $\pi_1$ is $E^0$-controlled, that is, there is a component $D$ of $E^0$ that is transverse to the center $Y_0$. Let us call $D$ the {\em control component}.

\begin{proposition}
\label{prop:controlcomponent} Assume that the first blowing-up in the resolution sequence $\mathcal S$ is monoidal. Then, the control component $D$ is dicritical.
\end{proposition}
\begin{proof} Up to skip the other components of $E^0$, we may assume that $E^0=D$. Let us find a contradiction with the fact that $D$ is invariant. We know that the first blowing-up $\pi_1$ is a dicritical blowing-up centered at a curve $Y_0$. Consider $\mathcal G_1=\mathcal F_1\vert_{E^1_1}$, where $E^1_1=\pi_1^{-1}(Y_0)$ and let $D'$ be the strict transform of $D$ by $\pi_1$. Since there are no compact curves in $\operatorname{Sing}(\mathcal F_1)$, we see that the generic points of the fiber $\pi_1^{-1}(\mathbf 0)= D'\cap E^1_1$ are regular for $\mathcal F_1$. Since $D'$ is invariant, we conclude that $\pi_1^{-1}(\mathbf 0)$ is invariant for $\mathcal G_1$.

In order to apply Proposition \ref{prop:vertical2}, let us show that $\mathcal G_1$ is non-singular. The only possible singularities should be a finite set of points in the fiber $\pi_1^{-1}(\mathbf 0)$, where $\mathcal G_1$ is a cart-wheel foliation. Since the fiber is invariant and it has zero self-intersection inside $E^1_1$, this set of points is empty. We conclude that $\pi_1$ is a vertical blowing-up and no singularities appear, this contradicts Proposition \ref{prop:verticalblowingups}, since the center $Y_0$ is invariant.
\end{proof}

\subsection{The General Case} Consider a radial foliated space
$$
\mathcal M_0=((\mathbb C^3,\mathbf 0),E^0,\mathcal F_0).
$$
Let us skip all the irreducible components in $E^0$, unless the control component of $\mathcal S$ is an adjusted resolution sequence starting at a monoidal blowing-up. We get a new radial foliated space:
$$
\widetilde{\mathcal M}_0=((\mathbb C^3,\mathbf 0),\tilde E^0,\mathcal F_0)
$$
such that $\mathcal S$ induces an adjusted resolution sequence $\tilde{\mathcal S}$ for $\widetilde{\mathcal M}_0$ and we have that $\tilde E^0=\emptyset$ except for the case when $Y_0$ is a germ of curve. In this case, the divisor $\tilde E^0$ has a single irreducible component, it is transverse to $Y_0$ and it is dicritical, in view of Subsection \ref{Dicriticalness of a Transverse Component}. 

\begin{lema} 
\label{lema:openbookgeneral}
In the above situation, we have that:
\begin{enumerate}
\item  There are local coordinates $x,y,z$ such that $\mathcal F_0$ is the open book foliation given by $ydz-zdy$.
\item The resolution sequence $\mathcal S$ satisfies that the blowing-ups 
$$
\pi_1,\pi_2,\ldots,\pi_{N-1}
$$ are quadratic blowing-ups centered at the infinitely near points $P_j\in M_j$  of the curve $y=z=0$.
\item  The last blowing-up $\pi_N$ is a monoidal blowing-up. The center $Y_{N-1}$ is the strict transform of the curve $y=z=0$. 
\item  The divisor $\tilde E^{N-1}$ has a single component trough $P_{N-1}$ and it is transverse to $Y_{N-1}$.
\end{enumerate}
\end{lema}
\begin{proof}
The resolution sequence $\tilde{\mathcal S}$ belongs to one of the cases A) or B). Then, we obtain that the foliation $\mathcal F_0$ is a open book foliation.

Now, we have two cases: $N=1$ or $N \geq 2$.

If we have that $N=1$, the blowing-up $\pi_1$ is necessarily a monoidal blowing-up centered at the singular locus $Y_0=(y=z=0)$ of $\mathcal F_0$; otherwise we do not destroy the whole singular locus. Note that in this case $\tilde E^0$ coincides with the control component, hence it is non-empty and transverse to $Y_0$.

Assume that $N \geq 2$. The first blowing-up $\pi_1$ is a quadratic blowing-up centered at the origin $P_0=\mathbf{0}$. The other possibility is to be equal to the singular locus $y=z=0$, but this would eliminate completely the singularities, contradicting that $N \geq 2$. Now, we can take local coordinates $x_1,y_1,z_1$ at $P_1$ such that the foliation $\mathcal F_1$ is given by 
$$
y_1dz_1-z_1dy_1
$$
and the exceptional divisor $\tilde E^1$ is locally given by $x_1=0$. We re-start the situation at $P_1$ and we end by induction on $N$.
\end{proof}

\begin{lema}
\label{lema: rectificaciondicritical}
Take local coordinates $x,y,z$ such that $\mathcal F_0$ is the open book foliation given by $ydz-zdy$. If $D$ is a dicritical component of $E^0$, then $D$ is transverse to the axis $y=z=0$.
\end{lema}
\begin{proof} Assume by contradiction that $D$ is not transverse to $y=z=0$. Then, one of the following two situations holds:
	\begin{enumerate}[i)]
		\item $(y=z=0)\subset D$.
		\item $(y=z=0)\not\subset D$. In this case, we have that  $N\geq 2$ and the first infinitely near point $P_1$ of $y=z=0$ belongs to the strict transform $D'$ of $D$ by $\pi_1$.
	\end{enumerate}
Assume that we are in case i). The property $(y=z=0)\subset D$ is stable under the quadratic blowing-ups $\pi_1,\pi_2,\ldots,\pi_{N-1}$. Thus, without loss of generality, we may assume that $N=1$ and $\pi_1$ is the monoidal blowing-up centered at $y=z=0$. When we perform the blowing-up we get a dicritical corner, this is not possible.
	
Assume now, that we are in case ii). In this case $D'\cap \pi_1^{-1}(\mathbf 0)$ defines a projective line $\Gamma$ passing through $P_1$. This curve $\Gamma$ is invariant in view of Proposition \ref{prop:radialinprojectiveplanes}. On the other hand, the curve $\Gamma$ is the intersection of two dicritical irreducible components of the divisor and the foliation is regular and simple at the points of $\Gamma$ different from $P_1$. This contradicts the fact that $\Gamma$ is invariant.
\end{proof}

Let us end the proof of Theorem \ref{main}. The question is if we can ``adapt'' the coordinates $x,y,z$ to the given divisor $E^0$ in such a way that $\mathcal F_0$ is given by $ydz-zdy=0$ and $E^0\subset (xyz=0)$.

In view of Lemma \ref{lema: rectificaciondicritical}, we can make a coordinate change $x\mapsto \phi(x,y,z)$ in such a way that the (only) possible dicritical component of $E^0$ is $x=0$. The invariant components (there are at most two of them) are necessarily of the type
$$
\lambda y+\mu z=0.
$$
Then, up to a linear change of coordinates in $y,z$, we are done.

\section{Almost Radial Foliated Spaces. Examples}
In this Section we give some examples of almost radial foliated spaces in dimension three that are not radial foliated spaces. We end the paper by showing that any almost radial foliated space $\mathcal M=(({\mathbb C^3,\mathbf 0}),E,\mathcal F)$ has at least one invariant germ of hypersurface.
\subsection{Non-Radial Open Books}
Take the foliation $\mathcal F$ on $(\mathbb C^3,\mathbf 0)$ given by $ydz-zdy=0$ and the divisor $E^0$ defined by 
$$
xy-z=0.
$$
The foliated space $(({\mathbb C^3,\mathbf 0}),E^0,\mathcal F)$ is almost radial but not radial. Indeed, the blowing-up centered at the axis $y=z=0$ gives an adjusted resolution sequence for it, hence it is an almost radial foliated space. It is not radial since $E^0$ is a dicritical component which is not transverse to the axis, contradicting Lemma \ref{lema: rectificaciondicritical}.
\subsection{Some Examples with Rational First Integral} Let us consider the foliations $\mathcal R_1$, $\mathcal R_2$ and $\mathcal R_3$ on $(\mathbb C^3,\mathbf 0)$ defined respectively by the rational differential $1$-forms
$$
\eta_i=d\phi_i,\quad i=1,2,3,
$$
where the rational functions $\phi_1$, $\phi_2$ and $\phi_3$ are given by
$$
\phi_1=\frac{xz^2+y^2}{yz},\quad  \phi_2=\frac{xy+z^2}{y},\quad \phi_3=\frac{xz^2+y^2}{z^2}.
$$
\begin{proposition}
  The foliated spaces $\mathcal N_i=(({\mathbb C^3,\mathbf 0}),\emptyset,\mathcal R_i)$ are almost radial, but not radial, for $i=1,2,3$.
\end{proposition}
\begin{proof} 
Let us see first that they are not radial foliated space, since the corresponding foliations $\mathcal R_i$ are not open book foliations.

We have that $(x=y=0)\cup (y=z=0)\subset \operatorname{Sing}(\mathcal R_1)$, then the singular locus of $\mathcal R_1$ does not correspond to an open book foliation. Concerning $\mathcal R_2$, we have that
$$
(y=z=0)\subset \operatorname{Sing}(\mathcal R_2),
$$
but the plane $x=0$ does not determine a Mattei-Moussu section, as it should be in the open book case (see Corollary \ref{cor:libroabierto}). Indeed, the curve $x=z=0$ is a curve of tangencies for this section. Finally, in the case of $\mathcal R_3$, we have a similar situation to the one in the case of $\mathcal R_2$.

Now, the foliated space $\mathcal N_1$ has a resolution sequence of length two, with evident centers. The foliated space $\mathcal N_2$ has a resolution sequence of length two, the first blowing-up centered at $y=z=0$ and the second one centered in a new curve étale over $y=z=0$. Finally, the foliated space $\mathcal N_3$ has a resolution sequence of length one centered at $y=z=0$.
\end{proof}
\subsection{Existence of Invariant Hypersurface} The existence of invariant hypersurface is a general question in the case of dicritical germs of codimension one foliations. The reader can look at \cite{Jou}, \cite{Can-C}, \cite{Can-Mat}, \cite{Can-Mol}, \cite{Can-MS} and others. The answer is positive for almost radial foliated spaces:
\begin{theorem}
 \label{teo:existencia de hipersuperficie invariante}
 Let $\mathcal M=(({\mathbb C^3,\mathbf 0}),E,\mathcal F)$ be an almost radial foliated space. Then there is at least one germ of invariant analytic surface through the origin.
\end{theorem}
\begin{proof} If $\mathcal M$ is radial, we are done, since the foliation is an open book foliation. It remains to consider the case when $\mathcal M$ is almost radial, but not radial. Moreover, we can assume that $E=\emptyset$, without loss of generality. Let
$$
\pi_1:\mathcal M_1=(M_1,E^1,\mathcal F_1)\rightarrow (({\mathbb C^3,\mathbf 0}),\emptyset,\mathcal F)
$$
be a monoidal blowing-up centered at a germ of curve $(Y_0,\mathbf 0)\subset ({\mathbb C^3,\mathbf 0})$ that is the first blowing-up of a resolution sequence for the almost radial foliated space $(({\mathbb C^3,\mathbf 0}),\emptyset,\mathcal F)$.

We have that $L=\pi_1^{-1}(\mathbf 0)$ is either invariant or not invariant for $\mathcal F_1$.

If $L$ is not invariant for $\mathcal F_1$, we are done by considering a regular point for $\mathcal F_1$ in the fiber $L$  and a germ of invariant surface on it.

Assume that $L$ is invariant for $\mathcal F_1$. Let $\mathcal G_1$ be the restriction of $\mathcal F_1$ to the exceptional divisor $E^1=E^1_1=\pi_1^{-1}(Y_0)$. We know that $\mathcal G_1$ defines a radial foliated space over $E^1_1$. Noting that $L$ is not contained in the singular locus of $\mathcal F_1$, the fiber
$L$ is invariant for $\mathcal G_1$. Recalling that the self-intersection of $L$ in $E^1_1$ is zero and that the only singularities of $\mathcal G_1$ are of cart-wheel type, we conclude that $\mathcal G_1$ has no singular points. Hence, no point in $E^1_1$ can be modified, except if the next bowing-up is a germ of curve $Y_1$ contained in $E^1_1$ and transverse to the fiber.

We repeat our argument with the new center $Y_1$. If at one step of the procedure we create a non-invariant fiber, we are done. If the fibers are invariant at each step, in the last one we can apply Proposition \ref{prop:verticalblowingups} to obtain the contradiction that the center is not invariant.
\end{proof}

Let us note that the above proof shows that there are infinitely many invariant surfaces for an almost radial germ of foliated space in dimension three.

\end{document}